\documentclass[11pt]{amsart}
\usepackage{amssymb, amsmath}
\usepackage{graphicx}
\usepackage{cite}
\usepackage{subfig, tikz}
\usetikzlibrary{decorations.pathmorphing}
\usepackage{graphics}
\usepackage{hyperref} 
\hypersetup{colorlinks=true, linkcolor=blue, citecolor = blue, urlcolor=blue, filecolor=magenta}

\newcommand{\td}{\text{d}}

\usepackage{bm}
\usepackage{enumerate}
\theoremstyle{plain}
\newtheorem{theorem}{Theorem}[section]
\newtheorem{lemma}[theorem]{Lemma}
\newtheorem{prop}[theorem]{Proposition}

\newtheorem{conjecture}[theorem]{Conjecture}

\theoremstyle{definition}
\newtheorem{remark}[theorem]{Remark}
\newtheorem{definition}[theorem]{Definition}

\newenvironment{myproof}[2] {\paragraph{\it Proof of {#1} {#2}.}}{\hfill$\square$}

\def\be{\begin{equation}}
\def\ee{\end{equation}}
\def\bea{\begin{eqnarray}}
\def\eea{\end{eqnarray}}

\begin{document}
\title[Existence and uniqueness of asymptotically flat toric gravitational instantons]{Existence and uniqueness of asymptotically flat toric gravitational instantons}
\author{Hari K. Kunduri}
\address{Department of Mathematics and Statistics\\
		Memorial University of Newfoundland\\
		St John's NL A1C 4P5, Canada}
	\email{hkkunduri@mun.ca}
\author{James Lucietti}
\address{School of Mathematics and Maxwell Institute for Mathematical Sciences\\
		University of Edinburgh\\
		King's Buildings, Edinburgh, EH9 3JZ, United Kingdom}
	\email{j.lucietti@ed.ac.uk}


\thanks{H. K. Kunduri  and J. Lucietti acknowledge the support of the  NSERC Grant  RGPIN-2018-04887 and the  Leverhulme Trust  Research Project Grant RPG-2019-355 respectively.}

\begin{abstract}
We prove uniqueness and existence theorems for four-dimensional asymptotically flat, Ricci-flat,  gravitational instantons with a torus symmetry.  In particular, we prove that such instantons are uniquely characterised by their rod structure, which is data that encodes the fixed point sets of the torus action. Furthermore, we establish that for every admissible rod structure there exists an instanton that is smooth up to possible conical singularities at the axes of symmetry.
The proofs involve adapting the methods that are used to establish black hole uniqueness theorems, to a harmonic map formulation of Ricci-flat metrics with torus symmetry, where the target space is  directly related to the metric (rather than auxiliary potentials). 
We also give an elementary proof of the nonexistence of asymptotically flat toric half-flat instantons. Finally, we derive a general set of identities that relate asymptotic invariants such as the mass to the rod structure.
\end{abstract}

\maketitle

\section{Introduction}

A gravitational instanton is a four-dimensional complete Riemannian manifold $(M, \mathbf{g})$ that is a solution to the Einstein equations~\cite{Gibbons:1979xm}. The physical motivation for studying such geometries originated in the context of euclidean quantum gravity, where they arise as saddle points in the gravitational path integral~\cite{Gibbons:1976ue}.  In the vacuum theory, so the metric is Ricci-flat,  one is interested in non-compact geometries that decay to a flat metric near infinity.  Two notable classes are the so-called asymptotically locally euclidean (ALE) and asymptotically locally flat (ALF) instantons. ALE spaces  approach $\mathbb{R}^4/\Gamma$ where $\Gamma$ is a finite subgroup of $SU(2)$, whereas ALF spaces approach  a circle bundle over $\mathbb{R}^3$.

Most work in this direction has been in the context of complete hyper-K\"ahler metrics, which are equivalently characterised as half-flat or as having a (anti)-self dual Riemann tensor. These are automatically Ricci-flat and hence consist of a special subclass of gravitational instantons.   A full classification of  ALE complete hyper-K\"ahler manifolds has been achieved \cite{Kronheimer:1989zs}.  More recently, a classification of ALF complete hyper-K\"ahler manifolds of cyclic type has been obtained \cite{Minerbe:2009gj}. It appears that much less is known about the classification of the more generic class of Ricci flat, but not necessarily hyper-K\"ahler, gravitational instantons.

 In this paper we will initiate the classification of asymptotically flat (AF) and Ricci-flat gravitational instantons. We will say a gravitational instanton $(M,\mathbf{g})$   is AF if it approaches a quotient of  $\mathbb{R} \times \mathbb{R}^3$ with the flat metric, under an isometry that acts as a translation in $\mathbb{R}$ and a rotation in  $\mathbb{R}^3$, where the metric on the $\mathbb{R}$ factor remains finite at infinity (we will define this precisely in Section \ref{sec:AF}). This class includes the notable example of the euclidean Schwarzschild metrics on $\mathbb{R}^2\times S^2$,
\be\label{Sch}
\mathbf{g}= \left(1-\frac{2m}{r} \right) \td \tau^2+ \frac{\td r^2}{1-\frac{2m}{r}} + r^2 (\td \theta^2+\sin^2\theta \td \phi^2) \; ,
\ee
where the parameter $m>0$ and the euclidean time $\tau\sim \tau + \beta$ must be periodically identified with period $\beta=8\pi m$ in order to obtain a complete metric, resulting in an asymptotically $S^1\times \mathbb{R}^3$ space. More generally, this class also includes the 2-parameter euclidean Kerr solution on $\mathbb{R}^2 \times S^2$ where now the identification is tilted $(\tau, \phi)\sim (\tau+ \beta, \phi+ \beta \Omega)$, with $\Omega$  related to the second parameter (the Schwarzschild solution is the $\Omega=0$ special case).   Note that our definition of AF overlaps with, but does not contain, the related definition of ALF manifolds as given by Minerbe~\cite{Minerbe}.  In particular, for the Kerr instanton, the vector field $\partial_\tau$ with asymptotically finite norm does \emph{not} have closed orbits; instead $\partial_\tau+ \Omega \partial_\phi$ is tangent to a circle with unbounded length at infinity.

For the ALF class of gravitational instantons,  Minerbe has defined a geometric invariant akin to the mass and  established a corresponding positive mass theorem~\cite{Minerbe} (more generally, this result holds whenever the Ricci tensor is non-negative, or $M$ is spin and the scalar curvature is non-negative - see also \cite{LSZ} for recent related work on ALF and ALG spaces).  The parameter $m$ appearing in \eqref{Sch} is precisely this geometric invariant. For AF instantons where the circle at infinity does not have finite length, such as the Kerr instanton,  there is as yet no analogous definition of a mass. In our analysis of AF gravitational instantons, we will identify an invariant in the asymptotic expansion around infinity that appears to capture the notion of a mass, although we do not address its positivity here. 

In the context of Lorentzian manifolds, the celebrated no-hair theorem states that the only asymptotically flat, stationary, vacuum spacetime that contains a black hole region, must belong to the Kerr family of solutions, for a review see~\cite{Chrusciel:2012jk}. Naturally, it was conjectured that this result would also hold in the Riemannian case, i.e., that the euclidean Kerr solution is the only AF gravitational instanton~\cite{Gibbons:1979xm}.  This question was investigated by Lapedes~\cite{Lapedes:1980st}, who showed that Israel's theorem \cite{Israel} -- which establishes that the Schwarzschild metric is the unique static black hole solution -- remains valid in Riemannian signature. However, Lapedes also pointed out that the uniqueness proof for Kerr black holes in the class of stationary and axisymmetric solutions is no longer valid in the Riemannian setting. Nevertheless, a Riemannian no-hair theorem  was still conjectured in that work~\cite{Lapedes:1980st}.

Remarkably, over 30 years later, Chen and Teo~\cite{Chen:2011tc, Chen:2015vva}, constructed a counterexample to the Riemannian no-hair conjecture. They found an explicit 2-parameter family of AF complete Ricci-flat metrics on $\mathbb{CP}^2 \backslash S^1$. The metrics are much more complicated than the Kerr solution (see Appendix \ref{app:explicit}), although both solutions do have a torus symmetry. 
Thus a natural problem presents itself: obtain a complete classification of AF gravitational instantons. In the absence of any extra structure this problem is at present out of reach. However, for solutions with a torus symmetry as above, the question is much more tractable. 

In this paper we will consider the classification of AF {\it toric} gravitational instantons, i.e.,  AF instantons with a torus symmetry (see Definition \ref{def:toric}).  In this case the notion of {\it rod structure}  can be defined for such instantons~\cite{Chen:2010zu} (we define this precisely in Section \ref{sec:axis}).  This is data that encodes how the torus action degenerates on the axes and determines the topology of the instanton, in particular, it fixes the 2-cycle structure and their (minimal) area. The Kerr and Chen-Teo metrics possess different rod structures so this offers a way to distinguish them. Indeed, our first main result shows that the rod structure provides a way to classify toric instantons.

\begin{theorem}
\label{th:unique}
There is at most one  AF toric gravitational instanton with a given rod structure.
\end{theorem}

We prove this by adapting the methods used for the uniqueness theorems for stationary vacuum black hole spacetimes with $D-3$ commuting axial symmetries, originally developed by Mazur in four dimensions~\cite{Mazur:1984wz, Hollands:2007aj,Hollands:2008fm}. Indeed, the notion of rod structure also plays a key role in the classification of such black holes, especially in higher dimensions. A key step in establishing the classic black hole uniqueness theorem is to show that the vacuum Einstein equations for stationary and axisymmetric spacetimes can be reformulated, in terms of an auxiliary `twist' potential, as a harmonic map with a Riemannian symmetric space target $SL(2, \mathbb{R})/SO(2)$. In the  case of Ricci-flat Riemannian manifolds with torus symmetry, one can repeat this reformulation in terms of twist potentials resulting in a harmonic map with a Lorentzian symmetric space target $SL(2, \mathbb{R})/SO(1,1)$; it is the indefinite nature of the target space metric which spoils the  uniqueness argument (this is essentially Lapedes's aforementioned observation concerning the failure of the uniqueness proof).  Nevertheless, we show that the method of Mazur can still be applied in the Riemannian case, albeit in a different way. In particular, we show that for Riemannian manifolds the Ricci-flat condition for toric metrics reduces {\it directly} to a harmonic map with Riemannian target $SL(2, \mathbb{R})/SO(2)$, without the need to introduce twist potentials. This then allows one to use the Mazur argument.\footnote{Such a direct reduction of stationary and axisymmetric vacuum solutions results in a harmonic map with Lorentzian target $SL(2, \mathbb{R})/SO(1,1)$, which cannot be used to prove uniqueness.}

Furthermore, using this harmonic map formulation we are also able to exploit previous results that allow one to address existence of solutions.  This leads to our second main result.

\begin{theorem}
\label{th:exist}
There exists a unique AF toric gravitational instanton for every admissible rod structure, that is smooth everywhere away from the axes.
\end{theorem}

For this we exploit a  theorem by Weinstein concerning the existence of harmonic maps with prescribed singularities~\cite{wein, wein2, Weinstein:2019zrh}. In this context the prescribed singularities correspond to the boundary conditions required in order to obtain a metric that is smooth at the axes.  Weinstein's theorem has been applied in the context of four and five-dimensional stationary and axisymmetric black hole spacetimes~\cite{wein, wein2, Khuri:2017xsc}, however it only depends on the character of the harmonic map, and therefore can be applied in our case. In particular, the existence proof reduces to exhibiting a so-called model map, which is a map that obeys the same boundary condition but need not be harmonic (although must have bounded tension).  Just like in the black hole case, this existence result cannot address regularity of the metric at the axes. In particular, generically one can have conical singularities on the axis components, and therefore we emphasise that our theorems do not address the classification of smooth AF toric instantons. Indeed, the Chen-Teo instanton shows that there are smooth instantons for rod structures distinct to the Kerr instanton, so it would appear reasonable to expect that other smooth instantons also exist.  In fact, given the rod structure can be arbitrarily complicated, in principle it is possible that there exists an infinite class of new smooth AF instantons. We will not address this interesting question in this paper. As in the case for four and five dimensional stationary and axisymmetric black hole spacetimes, we expect that methods from integrability theory can be employed to construct the general solution on the axes for any rod structure, and hence determine the (non)existence of metrics that are regular at the axes~\cite{Neugebauer:2011qb, Lucietti:2020ltw, Lucietti:2020phh}.

As mentioned above, a notable subclass of Ricci-flat gravitational instantons are (anti)-self dual. 
In fact, it has been argued by Lapedes that there are no AF half-flat instantons~\cite{Lapedes:1980st}. The argument, originally formulated by Gibbons and Pope in the context of asymptotically euclidean instantons, uses an index theorem for the Dirac operator~\cite{Gibbons:1979xn}.  We will give an elementary proof of this result in the context of toric instantons, thus establishing the following.

\begin{theorem}
\label{th:HF}
There are no  AF toric half-flat instantons.
\end{theorem}

The proof of this  first uses the well known fact that any toric half-flat metric is a Gibbons-Hawking metric. Then a global analysis of these geometries, which depend on a single axisymmetric harmonic function on $\mathbb{R}^3$,  shows that smoothness of the axes is incompatible with asymptotic flatness.

Our final result is a general set of identities for AF toric gravitational instantons, that relate asymptotic invariants such as the mass, to invariants defined by the rod structure, see Theorem \ref{th:id} for the precise statement. These can be thought of as Riemannian analogues of certain thermodynamic identities that hold for black holes~\cite{BCH}.  Indeed, the derivation of these identities, which involves integrating the harmonic map equation over the orbit space, was inspired by an analogous set of identities recently obtained for five-dimensional black holes~\cite{Kunduri:2018qqt}.

The organisation of this article is as follows. In Section \ref{sec:toric} we give the reduction of the Einstein equations for toric gravitational instantons, derive the general leading asymptotic expansion for AF instantons,  introduce the notion of rod structure and derive the geometry of the axis, and prove Theorem \ref{th:HF} on the nonexistence of the half-flat instantons. In Section \ref{sec:UEtheorems} we present our main uniqueness and existence results, in particular we prove Theorems \ref{th:unique} and \ref{th:exist}. In Section \ref{sec:id} we derive a general set of identities for AF toric instantons, summarised in Theorem \ref{th:id}.  We relegate some of the curvature calculations to Appendix \ref{app:curvature} and verify our identities for the known instantons in Appendix \ref{app:explicit}.

\subsection*{Acknowledgements} We thank Gary Gibbons, Marcus Khuri and Stefan Hollands for useful comments and discussions.

\section{Toric gravitational instantons}
\label{sec:toric}

\subsection{Reduction of Einstein equations} We begin by defining the class of instantons that we will be concerned with in this paper.

\begin{definition}
\label{def:toric}
A {\it toric gravitational instanton} is a four-dimensional, simply-connected, Ricci-flat, complete Riemannian manifold $(M,\mathbf{g})$, with an isometric effective torus  $T \cong U(1)^2$ action with at least one fixed point. 
\end{definition}

Now, for simplicity we will make the following technical assumption throughout this paper:  the $T$-action has no points in $M$ with a discrete isotropy subgroup.  It then follows, under our stated assumptions, that the orbit space of $M$ under such a torus action, which we denote by $\hat{M} := M/T$,  is a 2-dimensional simply-connected manifold with boundaries and corners~\cite[Proposition 1]{Hollands:2007aj} (see also the classic work~\cite[Theorem 1.12]{OR}).  We expect that the assumption that there are no discrete isotropy groups can be justified from asymptotic flatness as in the case of stationary and axisymmetric spacetimes~\cite{Hollands:2008fm}, although we will not pursue this here.

Let $\eta_i$, $i=1,2$ denote the Killing vector fields generating a torus action and define the $2\times 2$ Gram matrix $g$ of Killing fields by 
\be
g_{ij}:= \mathbf{g} (\eta_i, \eta_j).
\ee
Then, the interior of $\hat{M}$ corresponds to the points where $g_{ij}$ is rank-2, boundaries of $\hat{M}$ occur precisely at points where $g_{ij}$ is rank-1, and the corners where $g_{ij}$ is rank-0.

Now, consider the functions $\omega_i:= \star (\eta_1 \wedge \eta_2 \wedge \td \eta_i)$, 
where $\star$ is the Hodge dual with respect to $\mathbf{g}$, which must be constant by Ricci-flatness of $(M,\mathbf{g})$. Define the axis set
\be
\mathcal{A}:= \{ p\in M\, | \,  \det g(p) =0 \}  \; , 
\ee
which corresponds to the boundaries and corners of $\hat{M}$. Now, since the $T$-action has a fixed point the axis is non-empty and hence the constants $\omega_i=0$ for $i=1,2$. We deduce, by Frobenius' integrability theorem, that the distribution orthogonal to $\text{span}(\eta_1, \eta_2) \subset TM$ is integrable at every point.  Thus, on $M \backslash \mathcal{A}$ we can introduce coordinates $\phi^i$, with  $i=1,2$,
such that $\eta_i= \partial_{\phi^i}$ in terms of which the metric takes the block-diagonal form
\be
\mathbf{g}= g_{ij} \td \phi^i \td \phi^j+ \hat{g}   \; ,   \label{metric}
\ee
where $\hat{g}$ is a Riemannian metric on the orthogonal 2-dimensional surfaces which can be identified with the interior of $\hat{M}$. The Gram matrix $g_{ij}$ can then be thought of a matrix of functions on $\hat{M}$. 

Now, Ricci-flatness of $(M,\mathbf{g})$  is equivalent to the following equations on the orbit space $(\hat{M},\hat{g})$:
\bea
&&\td \hat{\star}( \rho J)=0 \; , \label{geq} \\
&&\hat{R}_{ab} = \hat{\nabla}_a \hat{\nabla}_b \log \rho + \frac{1}{4} \text{Tr} \left( J_a  J_b \right) \; ,\label{2dRic}
\eea
where $\hat{\star}$ is the Hodge dual, $\hat{\nabla}$ the metric connection and $\hat{R}_{ab}$ the Ricci tensor, all with respect to $\hat{g}$,  and 
\be 
J:= g^{-1} \td g , \qquad \rho:= \sqrt{ \det g } \; ,   \label{defs}
\ee
are a matrix valued 1-form  and non-negative function on $\hat{M}$, respectively.  It is worth noting that $J$  satisfies the zero curvature condition
\be
\td J+ J \wedge J=0  \; . \label{zerocurvature}
\ee
This reduction is well-known in the case of stationary and axisymmetric vacuum spacetimes; we provide our own self-contained derivation for the Riemannian case in  Appendix \ref{app:curvature}.

The trace of (\ref{geq}) implies $\rho$ is a harmonic function on $\hat{M}$. Therefore, we can introduce its harmonic conjugate $z$ via $\td z= -\hat{\star} \td \rho$,  so that $\td \rho\wedge \td z$ is positive orientation. It follows that wherever $\td \rho \neq 0$ we can use $(\rho, z)$ as a chart on the orbit space, so that
\be
\hat{g} = e^{2\nu} (\td \rho^2+ \td z^2)  \; ,   \label{orbit}
\ee
where $\nu$ is a function of $(\rho, z)$. 
In fact, one can show that this provides a global chart on the interior of $\hat{M}$. The proof of this, which uses asymptotic flatness, will be outlined in Section \ref{sec:AF} and is essentially the same as for stationary and axisymmetric vacuum spacetimes~\cite{wein, Chrusciel:2008js, Hollands:2008fm}.
In these coordinates (\ref{2dRic}) reduces to the first order equations for $\nu$,
\be
\partial_z \nu = \frac{\rho}{4} \text{Tr} ( J_z J_\rho), \qquad \partial_\rho \nu = -\frac{1}{2\rho}+ \frac{\rho}{8} \text{Tr}( J_\rho^2- J_z^2)  \; .   \label{dnu}
\ee
The integrability condition for $\nu$ is equivalent to (\ref{geq}). To see this it is useful to note that (\ref{geq}) and (\ref{zerocurvature}) in Weyl coordinates reduce to 
\bea
&&\partial_\rho(\rho J_\rho)+ \partial_z(\rho J_z)=0  \; ,  \label{geqweyl} \\
&& \partial_\rho J_z - \partial_z J_\rho+ [J_\rho, J_z]=0  \; .
\eea
Thus the Einstein equations for this class of metrics reduce to solving (\ref{geq}) or simply (\ref{geqweyl}).
It is worth noting that (\ref{geqweyl}) and the first order equations for $\nu$ (\ref{dnu}) take the same form as in the Lorentzian case of stationary and axisymmetric solutions (where $g_{ij}$ is Lorentzian instead of Riemannian as here).

\subsection{Asymptotic flatness}
\label{sec:AF}

We will now define a suitably general notion of asymptotically flatness that captures the gravitational instantons that we are interested in.

Consider the flat Riemannian manifold  $(M_{\flat}, \mathbf{g}_{0})$ defined as the quotient   $M_\flat :=[\mathbb{R} \times(\mathbb{R}^3\backslash B_R)]/\mathbb{Z}$, where $B_R\subset \mathbb{R}^3$ is a 3-ball of euclidean radius $R>0$ and $\mathbb{Z}$ is generated by the automorphism $(\tau, r, \theta, \phi) \mapsto (\tau + \beta, r,  \theta, \phi+ \beta\Omega)$,  $\beta>0, \Omega$ are constants,
\begin{equation}\label{model}
\mathbf{g}_0 = \td \tau^2 + \td r^2 + r^2 (\td \theta^2 + \sin^2\theta \td \phi^2) \; ,
\end{equation}
where $\tau \in \mathbb{R}$ and $(r, \theta ,\phi)$ are  standard spherical coordinates on $\mathbb{R}^3$. We will consider toric gravitational instantons that are modelled on $(M_{\flat}, \mathbf{g}_{0})$ near infinity.    This  of course includes the Schwarzschild solution for $\Omega=0$ (see (\ref{Sch})) and the Kerr solution for $\Omega \neq 0$ (see Appendix \ref{app:explicit}).   Observe that both $\frac{\partial}{\partial \tau}$ and $\frac{\partial}{\partial \phi}$ are Killing vector fields on $(M_{\flat}, \mathbf{g}_0)$, and in particular $\frac{\partial}{\partial \tau}$ has bounded norm. We can think of these Killing vectors as euclidean analogues of time translation and rotation respectively.    

Before giving our precise definition of asymptotic flatness, it is helpful to note that the Gram matrix of Killing fields for $\mathbf{g}_0$, in the basis $(\partial_\tau, \partial_\phi)$, is simply
\be
g_0 = \left( \begin{array}{cc} 1 & 0 \\ 0 & r^2 \sin^2\theta \end{array} \right)   \label{g0}
\ee
and  therefore the $(\rho,z)$ coordinates for this flat geometry  are 
\be
\rho = r \sin \theta, \qquad z= r \cos \theta   \; .   \label{polar}
\ee
with $\nu=0$.  It is convenient to use such polar coordinates $(r,\theta)$ to develop the asymptotic expansion for any AF metric.

\begin{definition}
\label{def:AF}
An \emph{asymptotically flat (AF) toric gravitational instanton}  is a toric gravitational instanton  with an end diffeomorphic to  $M_\flat$ for  some choice of $\beta$ and $\Omega$. Furthermore, on the end 
\be
\mathbf{g}= V (\td \tau+ \omega \td \phi)^2+ V^{-1} r^2\sin^2\theta \td \phi^2+ e^{2\nu} (\td r^2+ r^2 \td \theta^2) \; ,  \label{AFmetric}
\ee
where the coordinates $(\tau, r, \theta, \phi)$ are defined by the diffeomorphism, $\partial_\tau, \partial_\phi$ are the Killing fields that generate the torus action,  and
 the metric functions
\be
V= 1+ O_2(r^{-1}), \qquad \omega= O_2(r^{-1}), \qquad \nu= O_2(r^{-1})  \; ,
\ee
as $r \to \infty$ (here $f=O_n(r^{-1})$ if in cartesian coordinates $\partial^k f= O(r^{-1-k})$ for $k \leq n$).
\end{definition}

This definition ensures the metric $\mathbf{g}$ approaches $\mathbf{g}_0$ at infinity.
Now, observe that Definition \ref{def:AF} implies that relative to the basis $(\partial_\tau, \partial_
\phi)$ the function $\rho = r \sin\theta$, which gives (up to an additive constant) $z= r\cos\theta$, i.e. $(\rho, z)$ are related to the coordinates in the AF end by (\ref{polar}) for any AF instanton (not just flat space).  It follows that $(\rho, z)$ are well-defined coordinates in the AF end for any AF instanton.  This can be used to show  the much stronger statement that $(\rho, z)$ is  a global chart on the interior of the orbit space. The argument given in~\cite{wein} applies essentially unchanged, which we repeat now for completeness.
  
For an AF metric the orbit space metric must approach that of $\mathbf{g}_0$, i.e. as $r\to \infty$ we  have
\be
\hat{g} =(1+ O_2(r^{-1})) \left(\td r^2+ r^2 \td \theta^2 \right)  \; .
\ee
Thus  $\rho$ is asymptotic to the distance from the axis $\theta=0, \pi$.  Now consider the region $\{ 0< \rho<\rho_0 \} \subset \hat{M}$. For large enough $\rho_0$ its boundary $\rho=\rho_0$ is a smooth curve in the asymptotic region and the complement $\rho>\rho_0$ is simply connected.  Therefore, since $\hat{M}$ is simply connected, it follows that $\{ 0< \rho<\rho_0 \}$ must also be simply connected. By the Riemann mapping theorem we can conformally map this region to the strip $\{ 0< \Im \zeta < \rho_0 \} \subset \mathbb{C}$ such that $\rho=0$ becomes $\Im \zeta=0$ and $\rho=\rho_0$ becomes $\Im \zeta= \rho_0$. Now, since $\rho$ is harmonic on ($\hat{M}, \hat{g})$, under the conformal map it is also harmonic on the  strip, and hence  applying the maximum principle to $\rho - \Im \zeta$ on the strip shows that in fact $\rho=\Im \zeta$. It follows that $\rho$ has no critical points in the region $\{ 0< \rho<\rho_0 \}$ and since we can make $\rho_0$ as large as we like $\rho$ has no critical points in $\hat{M}$. Therefore, the $(\rho, z)$ can be used as a global coordinate chart in the interior of the orbit space.   We may now define global polar coordinates $(r, \theta)$ by (\ref{polar}) in terms of which the metric is (\ref{AFmetric}).

The Gram matrix of $(\partial_\tau, \partial_\phi)$ for (\ref{AFmetric}) is
\be
g = \begin{pmatrix} V & V \omega \\ V \omega & \rho^2 V^{-1} + V \omega^2 \end{pmatrix}  \; ,   \label{Vom}
\ee
and in this parameterisation the equations (\ref{geq}) reduce  to the system
\begin{align}
&\hat{\nabla}^a (\rho \hat{\nabla}_a \log V) = \frac{V^2}{\rho} \hat{\nabla}^a \omega \hat{\nabla}_a  \omega   \label{Veq} \\
&\hat{\nabla}^a \left( \frac{V^2}{\rho} \hat{\nabla}_a \omega \right)=0  \; ,  \label{omeq}
\end{align}
on the orbit space $(\hat{M}, \hat{g})$.   It will be useful to reformulate this system in terms of the Ernst potential $W$ defined by
\begin{equation}
\td W = -\frac{V^2}{\rho} \hat \star \td \omega  \; ,  \label{Wdef}
\end{equation} 
which exists globally since the orbit space is simply connected.
The system (\ref{Veq}), (\ref{omeq}) now reads
\begin{align}
&\hat{\nabla}^a( \rho  \hat{\nabla}_a \log V)= \rho V^{-2} \hat{\nabla}^a W \hat{\nabla}_a W \label{Veqalt} \\ 
& \hat \nabla^a \left[\frac{\rho}{V^2} \hat\nabla_a W \right] =0  \label{Weq}  \; .
\end{align}
We are now ready to state the main result of this section.

\begin{prop}
\label{prop:AF}
Consider an AF toric gravitational instanton. The Gram matrix of Killing fields in the basis $(\partial_\tau, \partial_\phi)$ has an asymptotic expansion
\be
g = \begin{pmatrix} 1- \frac{2m}{r} + O_2(r^{-2}) & \frac{2j\sin^2\theta}{r}+ O_2(r^{-2}) \\\frac{2j\sin^2\theta}{r}+ O_2(r^{-2} ) & r^2\sin^2\theta \left(1+ \frac{2m}{r} + O_2(r^{-2})\right) \end{pmatrix}  \label{AFg}
\ee
as $r \to \infty$, 
where $(r, \theta)$ are polar coordinates defined by (\ref{polar}) and $m, j$ define two asymptotic invariants. Furthermore,
\be
e^{2\nu}= 1 + \frac{2m}{r} + O_2(r^{-2}) \; . \label{AFnu}
\ee
\end{prop}

\begin{proof}
We will adapt a proof of Beig and Simon given in the context of asymptotically flat stationary spacetimes~\cite{BeigII}.
The system (\ref{Veqalt}), (\ref{Weq}) in the coordinates (\ref{orbit}) can be written as a Poisson like equations on $E_R =:\mathbb{R}^3 \backslash B_R$: 
\begin{align}
& \Delta \log V = V^{-2} \nabla W \cdot \nabla W, \label{VR3eq}  \\
& \Delta W= 2 V^{-1} \nabla V \cdot \nabla W \; ,  \label{WR3eq}
\end{align}
where $\nabla$ is the standard derivative on $\mathbb{R}^3$ in cylindrical coordinates $(\rho,z, \varphi)$ with $\varphi$ is an auxiliary $2\pi$-periodic angle,  $\Delta:=\nabla \cdot \nabla$ the corresponding Laplacian, and $V, W$ are axisymmetric functions on $\mathbb{R}^3$ (i.e. only functions of $(\rho, z)$).    Now, by Definition \ref{def:AF} and the definition of the Ernst potential (\ref{Wdef})  it follows that   
\be
W= O_2(r^{-2}).  \label{Wdecay}
\ee 
 Therefore, the r.h.s. of (\ref{VR3eq}) is $O_1(r^{-6})$.  By  \cite[Lemma B]{BeigII} we deduce that there is a constant $m$ such that
\begin{equation}
\log V = -\frac{2m}{r} + O_2(r^{-2}) , 
\ee
which establishes the  asymptotic expansion for $g_{\tau\tau}=V$.

Next, (\ref{Wdecay}) implies the r.h.s of (\ref{WR3eq}) is $O_1(r^{-5})$. Therefore,  \cite[Lemma A]{BeigII} together with axisymmetry of $W$,  implies there are constants $N, j$ such that
\begin{equation}
W = \frac{N}{r}  - \frac{2 j \cos\theta}{r^2}  + O_2 \left(\frac{\log r}{r^3}\right)  \; .
\end{equation} 
Clearly we must set  $N=0$ in order to satisfy the estimate  (\ref{Wdecay}) that follows from asymptotic flatness.  Next, substituting the expansions for $V,  W$ back into the r.h.s of (\ref{WR3eq})  we can write
\be
\Delta \left( W  - \frac{4m j \cos \theta}{r^3} \right) =  O_1 \left(\frac{\log r}{r^6}\right) \; ,
\ee
where we have used $\Delta (r^{-3}\cos\theta)= 4 r^{-3} \cos\theta$. Therefore, by \cite[Lemma B]{BeigII}  we deduce
\begin{equation}
W   =  - \frac{2j \cos\theta}{r^2}  + O_2 \left(\frac{1}{r^3}\right) \; ,  \label{Wasymp}
\end{equation}
where the monopole term is again absent due to (\ref{Wdecay}). Finally, integrating (\ref{Wdef}) with (\ref{Wasymp}) we find
\be
\omega = \frac{2 j \sin^2\theta}{r}+ O_2(r^{-2})  \; ,
\ee
where we have set integration constant such that $\omega \to 0$ at infinity. This therefore establishes the decay for $g_{\tau \phi}= V \omega$.

It remains to find the asymptotic expansion of the conformal factor $e^{2\nu}$. In terms of the coordinates (\ref{polar}) one finds (\ref{dnu}) imply
\begin{equation}\begin{aligned}
\partial_r \nu & = -\frac{m}{r^2}  + O_1(r^{-3}) \; , \\
\partial_\theta \nu & = O_1(r^{-2})  \; ,
\end{aligned}
\end{equation} which integrates to $\nu = m/r + O_2(r^{-2})$ and hence (\ref{AFnu}).
\end{proof}

It is worth emphasising that under the remaining coordinate freedom $z\to z+ c$ where $c$ is a constant, the asymptotic parameters $m, j$ are invariant. These invariants are euclidean analogues of the mass and angular momentum of stationary and axisymmetric spacetimes (hence the notation). We will therefore refer to $m, j$ as mass and angular momentum of AF toric instantons. As an example, consider euclidean Schwarzschild in the coordinates (\ref{polar}):
\begin{equation}
\mathbf{g}= \left( 1 + \frac{\mu}{2r}\right)^4 \left[ \td r^2 + r^2 \td \theta^2 \right]  + \frac{\left( 1 - \frac{\mu}{2r}\right)^2}{\left( 1 + \frac{\mu}{2r}\right)^2} \td \tau^2 +  \left( 1 + \frac{\mu}{2r}\right)^4 r^2 \sin^2\theta \td \phi^2  \; ,
\end{equation} which, comparing to our asymptotic form, shows $m=\mu$ and $j=0$. 
Similarly, for the euclidean Kerr instanton given in  Appendix \ref{app:explicit} and verify that $m, j$ correspond to euclidean versions of the mass and angular momentum parameters.

We emphasise that strictly speaking  we have not defined the asymptotic invariants $(m,j)$ in the more general context of AF gravitational instantons, although we expect a definition analogous to the ADM mass to exist. To this end, we find the following conjecture natural.
\begin{conjecture}
To any AF gravitational instanton $(M, \mathbf{g})$, we can assign a geometric invariant `mass' $m$ (which coincides with the invariant $m$ above for toric instantons) with the property that $m\geq 0$ with equality if and only if $(M, \mathbf{g})$ is flat.
\end{conjecture}

We expect the above can be established is a similar manner to the positive-mass theorem for ALF instantons~\cite{Minerbe}. It would be interesting to investigate this further.

\subsection{Axis and rod structure}
\label{sec:axis}

In the orbit space the axis $\mathcal{A}$ corresponds to the $\rho=0$ boundary of $\hat{M}$. As shown in~\cite{Hollands:2008fm}, this divides into boundary segments corresponding to intervals $$(-\infty, z_1), \quad (z_1, z_2),\quad  \dots, \quad (z_N, \infty),$$ with $z_1<z_2<\dots<z_N$, on which $g_{ij}$ is rank-1, that are separated by corners $z_A, A=1, \dots, N$ where $g_{ij}$ is rank-0. We denote the intervals, called rods, by $I_A=(z_{A-1}, z_A)$  for $A=1, \dots, N+1$, where it is understood that $z_0=-\infty$ and $z_{N+1}=\infty$. 

There are a number of invariants associated to this boundary. The length of each finite rod $I_A$, where $A=2, \dots, N$ is 
\be
\ell_A= z_{A}-z_{A-1}  \; .
\ee
The rod vector  $v_A$ associated to each rod $I_A$ is the unique up-to-scale Killing vector $v_A$ which vanishes on $I_A$ (recall $g$ is rank-1 on $I_A$). We fix the scale such that $v_A$ is a $2\pi$-periodic vector so we can write
\be
v_A= v_A^i \eta_i
\ee
where $\underline{v}_A:= (v^1_A, v^2_A) \in \mathbb{Z}^2$ are coprime integers and $\eta_i$ are $2\pi$-periodic generators of the torus action.  Note this does not fully fix $v_A$, which is now defined up to an overall sign.

\begin{definition}
\label{def:rod}
 The collection of boundary data
\be
\mathcal{R}:= \{ (I_A, \underline{v}_A)\, | \, A=1, \dots, N+1 \} 
\ee
is called the {\it rod structure} of $(M, \mathbf{g})$. If the rod vectors of consecutive rods $I_{A-1}, I_A$ satisfy the condition
\be
\det \begin{pmatrix}\underline{v}_{A-1},   \underline{v}_A \end{pmatrix} = \pm 1  \; .   \label{admissible}
\ee
the rod structure is said to be {\it admissible}.
\end{definition}

The rod structure is a fundamental invariant for such spaces.
The admissibility condition is required in order to avoid orbifold singularities at the corners~\cite{Hollands:2008fm}.
We will ultimately be interested in such admissible rod structures, although much of our analysis will not use this condition.

For AF geometries as defined above we can introduce a preferred basis of independently $2\pi$-periodic Killing fields as follows.  Define $(\psi, \chi)$ by 
\be
\tau= \hat{\beta} \psi, \qquad \phi = \chi+ \Omega \hat{\beta} \psi  \; ,
\ee
where $\hat{\beta}:= \beta/(2\pi)$, so that the identifications on $(\tau, \phi)$ are equivalent to $(\psi, \chi)\sim (\psi+2\pi, \chi)$ and $(\psi, \chi)\sim (\psi, \chi+2\pi)$.
This gives
\be
\partial_\psi=\hat{\beta} \left(  \partial_\tau+ \Omega \partial_\phi \right) , \qquad \partial_\chi = \partial_\phi  \; ,  \label{basis}
\ee
so the change of basis matrix is
\be
L=  \left( \begin{array}{cc} \hat{\beta} & \Omega\hat{\beta} \\ 0 & 1\end{array} \right)    \label{L}
\ee
and the Gram matrix of Killing fields is this new basis is $\tilde{g} = L g L^T$ where $g$ is relative to our original basis $(\partial_\tau, \partial_\phi)$. Note that in our definition of AF the parameter $\hat{\beta} \Omega$ is only defined up to an additive integer.  In turn, this implies the above $2\pi$-periodic basis is only defined up to $(\partial_\psi, \partial_\chi) \to (\partial_\psi + q \partial_\chi, \partial_\chi)$ where $q \in \mathbb{Z}$ (this preserves the admissibility condition (\ref{admissible})).

Relative to the $(\partial_\psi, \partial_\chi)$ basis the rod vectors of the semi-infinite rods are thus $\underline{v}_1= (0,1)$ and $\underline{v}_{N+1}=(0,1)$, which we assume henceforth.  For example, the euclidean Kerr solution has one finite rod $I_2$, so $N=2$, with rod vector $\underline{v}_2=(1,0)$. It is interesting to list all admissible rod structures for the simplest cases. For example, suppose there is one finite rod, with rod vector $\underline{v}_2=(p,q)$. Then the admissibility conditions $\det ( \underline{v}_1, \underline{v}_2)=\pm 1$  and $\det ( \underline{v}_2, \underline{v}_3)=\pm 1$ both give $p=\pm 1$. Therefore, since the rod vector is only defined up to a sign, the most general case is $\underline{v}_2=(1, q)$ where $q \in \mathbb{Z}$. Furthermore, using the freedom in the definition of the basis discussed above we can always set $q=0$, which corresponds to the Kerr rod structure.  This is depicted in Figure \ref{fig:onerod}.
\begin{figure}[h!]

\centering
\subfloat{
\begin{tikzpicture}[scale=1.2, every node/.style={scale=0.6}]
\draw[very thick](-4,0)--(-1.2,0)node[black,left=2.8cm,above=.2cm]{$(0,1)$};
\draw[thick](-1,0)--(1,0)node[black,left=1.9cm,above=.2cm]{$(1,0)$};
\draw[very thick](1.2,0)--(4,0)node[black,left=2.8cm,above=.2cm]{$(0,1)$};
\draw[fill=black] (-1.1,0) circle [radius=.1] node[black,font=\large,below=.1cm]{};
\draw[fill=black] (1.1,0) circle [radius=.1] node[black,font=\large,below=.1cm]{};
\end{tikzpicture}}
\caption{General rod structure with one finite rod is that of the Kerr instanton. \label{fig:onerod}}
\end{figure}
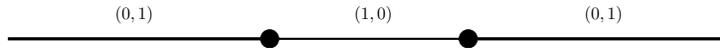

For two finite rods, so $N=3$, we find that the most general admissible rod structure is given by the rod vectors $\underline{v}_2=(1, q)$ and $\underline{v_3}=(1,q\pm 1 )$ where $q\in \mathbb{Z}$ and again we have used the overall sign freedom to fix the rod vectors. Again, using  the above freedom in the basis we can always set $q=0$, which corresponds to the Chen-Teo instanton~\cite{Chen:2011tc}. This is depicted in Figure \ref{fig:tworod}.

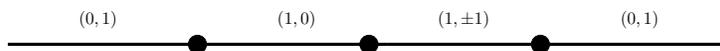
\begin{figure}[h!]
\subfloat{
\begin{tikzpicture}[scale=1.2, every node/.style={scale=0.6}]
\draw[very thick](-4,0)--(-2,0)node[black,left=2cm,above=.2cm]{$(0,1)$};
\draw[very thick](-1.8,0)--(0.2,0)node[black,left=2cm,above=.2cm]{$(1,0)$};
\draw[very thick](0.2,0)--(1.8,0)node[black,left=1.5cm,above=.2cm]{$(1,\pm 1)$};
\draw[very thick](2,0)--(4,0)node[black,left=2cm,above=.2cm]{$(0,1)$};
\draw[fill=black] (-1.9,0) circle [radius=.1] node[black,font=\large,below=.1cm]{};
\draw[fill=black] (0,0) circle [radius=.1] node[black,font=\large,below=.1cm]{};
\draw[fill=black] (1.9,0) circle [radius=.1] node[black,font=\large,below=.1cm]{};
\end{tikzpicture}}
\caption{General rod structure with two finite rods is that of the Chen-Teo instanton. \label{fig:tworod}}
\end{figure}

 Each component of the axis $I_A$  is defined by the zero-set of the Killing field $v_A$ and the metric is smooth at these fixed points sets provided two conditions are met: (i) the metric functions $g_{ij}$ are smooth functions of $(\rho^2,z)$; (ii) there are no conical singularities at $I_A$,  the condition for which is
\be
\lim_{\rho\to 0, z\in I_A} \frac{ \td |v_A|^2 \cdot \td |v_A|^2}{4|v_A|^2}=1 \; . \label{noconical}
\ee
We will not address condition (i) in this work, although we expect  it can be established as in the case of  four-dimensional stationary and axisymmetric black hole spacetimes~\cite[Theorem 5]{wein}. Condition (ii) typically imposes a constraint on the parameters of a solution, although for the semi-infinite axes this is automatically satisfied for AF solutions as above.

We will need to analyse the geometry near the axis in more detail.
To this end it is convenient to introduce a $2\pi$-periodic basis of Killing fields $(u_A, v_A)$,  adapted to $I_A$, where $v_A$ is the rod vector of $I_A$ and 
\be
L_A= \begin{pmatrix} u^1_A & u^2_A \\ v^1_A & v^2_A \end{pmatrix} \in GL(2, \mathbb{Z})  
\ee
defines an independent vector $u_A$ that is non-vanishing in the interior of $I_A$.  Then the Gram matrix relative to the adapted basis is $g'= L_A g L_A^T$ where $g$ is relative to our preferred basis $(\partial_\psi, \partial_\chi)$.  Furthermore, it is helpful to parameterise $g'$ over $z\in I_A$ as,
\be
g'= L \begin{pmatrix}  h+ \rho^2 h^{-1} w^2 & \rho^2 h^{-1} w \\ \rho^2 h^{-1} w & \rho^2 h^{-1}  \end{pmatrix} L^T \; , \label{axisg}
\ee
where $h>0$ and $w$ are smooth functions of $(\rho^2, z)$ at $\rho=0$. This parameterisation can be established by combining the facts that in the adapted basis $\underline{v}'_A=(0,1) \in \text{ker}(g')_{\rho, z \in I_A}$ and $\text{rank}(g')_{\rho, z \in I_A}=1$, together with the axis regularity conditions (i) and (ii) mentioned above (the factors of $L$ are required to ensure $\det g'= \hat{\beta}^2 \rho^2$).  Note that in this parameterisation $|v_A|^2= \rho^2 h^{-1}$.

 We may determine the geometry near the axes as follows. Using (\ref{axisg}) it can be shown that  (\ref{dnu}) as $\rho \to 0$ simplifies to $\partial_z \nu = -\partial_z h/(2h) +O(\rho^2)$ and $\partial_\rho \nu= O(\rho)$, which thus integrates to give
\be
e^{2\nu} = \frac{c_A^2}{h} +O(\rho^2)\; ,   \qquad z \in I_A,   \label{nuaxis}
\ee
where $c_A>0$ is a constant.  The metric near the axis rod $I_A$ then takes the form
\be
\mathbf{g}=  \frac{1}{h} \left[ (c_A^2 +O(\rho^2))\td \rho^2+ \rho^2 (\td \phi^2+ (w+\Omega)\hat{\beta} \td \phi^1)^2 \right]+ \left( \frac{c_A^2}{h}+O(\rho^2) \right)\td z^2+ \hat{\beta}^2 h (\td \phi^1)^2   \; ,\label{gnearaxis}
\ee
as $\rho\to 0$, where $\phi^i$ are $2\pi$-periodic coordinates adapted so $u_A=\partial_{\phi^1}, v_A=\partial_{\phi^2}$. Therefore,  for each $z\in I_A$ we see that there is a conical singularity at $\rho=0$ in the $(\rho, \phi^2)$ plane which is absent if and only if $c_A=1$. This also follows from (\ref{noconical}) together with (\ref{nuaxis}). Furthermore, it is now clear that the metric extends smoothly to $\rho=0$ as a consequence of $(h, w)$ being smooth functions of $(\rho^2, z)$ there (this can be seen by transforming to cartesian coordinates in the $(\rho, \phi^2)$ plane).

The finite axis rods $I_A$ correspond to 2-surfaces of $S^2$ topology $C_A$ in the spacetime. These are known as bolts~\cite{Gibbons:1976ue}. To see this, simply note that on the axis $\rho=0$ and $z \in I_A$, defined by the vanishing of $v_A$,  there is an independent axial vector $u_A$ that is non-vanishing in the interior of $I_A$ but must vanish at the end points of $I_A$ (since these correspond to corners of the orbit space). This defines a closed 2-surface with a $U(1)$-action with two fixed points which must be homeomorphic to $S^2$.   These surfaces are representatives of a basis of 2-cycles $[C_A] \in H_2(M)$.

We may determine the geometry of the axes as follows. 
From (\ref{gnearaxis}) deduce that the metric $\mathbf{g}_A:= \iota_A^* \mathbf{g}$ induced on $C_A$, where $\iota_A: C_A \to M$ is the natural inclusion, is simply
\be
\mathbf{g}_A =  \frac{\td z^2}{h(z)} + \hat{\beta}^2 h(z) (\td \phi^1)^2 \; . \label{bubble}
\ee
This coordinate system only covers the interior of $I_A$ and at the endpoints $h=0$ (since these are corners of the orbit space). The metric extends smoothly to $C_A \cong S^2$ if and only if
\be
h'(z_{A-1})=- h'(z_A)=\frac{1}{\hat{\beta}} \; .   \label{smoothcorner}
\ee
This latter condition ensures the absence of conical singularities at points corresponding to the corners of the orbit space.  We may now compute the area of $C_A$  from (\ref{bubble}) which is simply
\be
A[C_A] := \int_{C_A} \td\text{vol}(\mathbf{g}_A) ={\beta} \ell_A  \; .  \label{area}
\ee 
This provides a geometrical interpretation of the orbit space invariants $\ell_A$.

\subsection{Half-flat  instantons}

A four-dimensional Riemannian manifold is half-flat  (i.e.  the Riemann tensor is self-dual or anti-self-dual)  if and only if it is K\"ahler and Ricci-flat.  Thus half-flat gravitational instantons arise as a special case of the Ricci-flat instantons. In particular, we would like to classify toric half-flat instantons. In fact, long ago Gibbons and Ruback~\cite{Gibbons:1987sp} showed that any hyper-K\"ahler metric with two commuting Killing vectors must  be a Gibbons-Hawking metric~\cite{Gibbons:1979zt}. 
For completeness we give here a more explicit argument in the context of toric K\"ahler manifolds, i.e., K\"ahler manifolds admitting a holomorphic and isometric torus action.

\begin{theorem}
Any toric half-flat  gravitational instanton is a Gibbons-Hawking metric  determined by an axisymmetric harmonic function $H$ on $\mathbb{R}^3$.
\end{theorem}

\begin{proof}
It is well known that any toric K\"ahler metric can be written as (see e.g.~\cite{Abreu}),
\be
\mathbf{g}= F_{ij} \td y^i \td y^j+ F_{ij} \td \tilde{\phi}^i \td \tilde{\phi}^j,    \label{kahler}
\ee
where $F_{ij}= \partial_i \partial_j f$ with $\partial_i:= \partial_{y^i}$,  $f=f(y)$ is the K\"ahler potential,  $ \Omega= F_{ij} \td y^i \wedge \td \tilde{\phi}^j$ is the K\"ahler form,  $z^i= y^i + i \tilde{\phi}^i$ are holomorphic coordinates, and $i,j \in \{ 1,2 \}$. Here $\partial_{\tilde{\phi}^i}$ generate the torus symmetry, although for convenience we do not assume they have periodic orbits, i.e. they are some real linear combination of the $2\pi$-periodic Killing fields. In fact, since $(M, \mathbf{g})$ is simply connected the chart is globally defined away from the axis. This chart is defined up to a rigid $GL(2, \mathbb{R})$ transformation $ z^i \to A_{ij} z^j$ where $A\in GL(2, \mathbb{R})$ and a translation freedom $z^i \to z^i + c^i$ where $c^i\in \mathbb{R}$.  

From standard formulas for the Ricci tensor of a K\"ahler metric in holomorphic coordinates,
\be
R_{z^i \bar{z}^j}=\partial_{z^i} \partial_{\bar{z}^j} \log \det g_{z^i \bar{z}^j} = \partial_i \partial_j \log \det F_{ij} \; ,
\ee
where the second equality follows from invariance under the torus symmetry (i.e. $F_{ij}$ depends only on $y^i$). Therefore, the metric is Ricci-flat iff
\be
\log \det F_{ij}= a_i y^i+ b \; ,
\ee
where $a_i, b$ are constants. In order to have a non-empty axis we must have at least one $a_i\neq 0$ and therefore using the translation and $GL(2, \mathbb{R})$ freedom we may always set $\log \det F= 2 y^1$ (this will be a convenient choice below).

We wish to write the metrics  (\ref{kahler}) in the form (\ref{metric}) and the orbit space metric in the $(\rho, z)$ coordinates (\ref{orbit}). To this end, first note that the Gram matrix of Killing fields of (\ref{kahler}) is $F_{ij}$ and therefore it immediately follows that $\rho^2= \det F$. Therefore, by Ricci-flatness and the above coordinate choice we deduce
\be
\rho= \exp (y^1)  \; .
\ee
The orbit space metric for (\ref{kahler}) is $\hat{g}= F_{ij} \td y^i \td y^j$ and the harmonic conjugate to $\rho$ is given by
\be
\td z= - \hat{\star} \td \rho = \td (\partial_2 f) \; ,
\ee
where the second equality follows from a short computation using $\det F= \rho^2$, and we assume $\td y^1\wedge \td y^2$ is positive orientation. Therefore, we have determined the coordinate change $(y^1, y^2) \mapsto (\rho, z)$. Matching the orbit space metric to (\ref{orbit}) then reduces to 
\be
e^{2\nu} = H  \; ,  \label{HFnu}
\ee
where we define
\be
H := F_{22}^{-1}  \; , \qquad \chi := \frac{F_{12}}{F_{22}}  \; .
\ee
A straightforward calculation reveals that
\be
\rho \, \td H= - \hat{\star} \td \chi \; , \label{GH}
\ee
for any K\"ahler potential $f$, where we have  used $F_{11}= (\rho^2+F_{12}^2)/F_{22}$ (i.e. $\det F=\rho^2$) to simplify the r.h.s. of (\ref{GH}). Finally, in terms of $H, \chi$ the metric takes the form
\be
\mathbf{g}=  H^{-1}( \td \tilde{\phi}^2+ \chi \td \tilde{\phi}^1)^2 + H [ \rho^2 (\td \tilde{\phi}^1)^2+ \td \rho^2+ \td z^2]  \; ,  \label{GHmetric}
\ee
which together with (\ref{GH}) is precisely a Gibbons-Hawking metric.  Indeed, the metric in square brackets in (\ref{GHmetric}) is euclidean space $\mathbb{R}^3$ in cylindrical polar coordinates and (\ref{GH})  states that $H$ is an axisymmetric harmonic function on $\mathbb{R}^3$. 
\end{proof}

We are now ready to prove the main result of this section. \vspace{.2cm}

\begin{myproof}{Theorem}{\ref{th:HF}}
For a half-flat instanton comparing (\ref{HFnu}) with (\ref{AFnu}) we immediately deduce that  
\be
H= 1+ \frac{2m}{r}+ O_2(r^{-2}) \;.   \label{HAF}
\ee
Solving (\ref{GH}) then implies
\be
\chi = \chi_0+ 2 m \cos\theta+ O_2(r^{-1})\; ,
\ee
where $\chi_0$ is a constant.
Now, since (\ref{AFg})  shows that any linear combination of the Killing fields which is bounded at infinity must be proportional to $\partial_\tau$, comparing to (\ref{GHmetric}) we deduce that $\tilde{\phi}^2= \tau + c_0 \phi$ and $\tilde{\phi}^1=\phi$, where $c_0$ is some constant. Hence
\be
g_{\tau \phi} = H^{-1} (\chi+c_0)  = \chi_0+c_0+ 2 m \cos\theta+ O_2(r^{-1})
\ee
and therefore AF implies $m=0$ (and $c_0=-\chi_0$). 

Now, from (\ref{GHmetric})  we also deduce that $H^{-1}= g_{\tau\tau}$, so $H^{-1}$ is a smooth nonnegative function on $(M, \mathbf{g})$.  In particular, wherever $g_{\tau\tau}>0$ the function $H$ must be smooth and positive.  However, this also shows that $H$ is singular at any fixed points of $\partial_\tau$.  In fact, on the interior of any rod $I_A$ we have (\ref{nuaxis}) and therefore (\ref{HFnu}) also shows that $H$ is a smooth positive function on the interior of $I_A$.  Therefore, the fixed points of $\partial_\tau$ must correspond to the corners of the orbit space $z=z_A$, i.e. where $\partial_\psi=\partial_\chi=0$.  Thus $H$ may have isolated singular points at the corners of the orbit space and is otherwise smooth everywhere else.  

Next, from the $zz$ component of the metric induced on the interior of a finite rod $I_A$, which is given by (\ref{bubble}),  we deduce that  $H= h^{-1}$ on $I_A$ where $h(z)$ vanishes at the endpoinds. Hence smoothness at the endpoints (\ref{smoothcorner}) shows that
\be
H|_{\rho=0}= \frac{\hat{\beta}}{|z-z_A|} + O(1)
\ee
as $z \to z_A$, 
for any $A=1, \dots, N$.  Therefore, since $H$ is an axisymmetric harmonic function on $\mathbb{R}^3$, we deduce that $H$ must have simple poles at the corners and hence
\be
H =\frac{\hat{\beta}}{\sqrt{\rho^2+(z-z_A)^2}}+ O(1)  \; ,
\ee
where the $O(1)$ term is a harmonic function smooth at $(\rho,z)=(0,z_A)$.  Thus we can write
\be
H = H_0+ \sum_{A=1}^N \frac{\hat{\beta}}{\sqrt{\rho^2+(z-z_A)^2}}  \; ,
\ee
where $H_0$ is an everywhere smooth harmonic function on $\mathbb{R}^3$.  Now, by asymptotic flatness we have $H_0\to 1$ at infinity, hence $H_0$ is a bounded harmonic function everywhere smooth on $\mathbb{R}^3$, which must be a constant, so $H_0=1$.  Finally, comparing to (\ref{HAF}), we deduce that $m= N \hat{\beta}/2$ which for $N\geq 1$ clearly contradicts the fact asymptotic flatness requires $m=0$.
\end{myproof}

  \section{Uniqueness and existence}
  \label{sec:UEtheorems}
  
\subsection{Uniqueness theorem and Mazur identity}

The uniqueness problem for a class of nonlinear PDEs on the Riemannian manifold $(\hat{M}, \hat{g})$, where $\hat{g}$ is given globally in $(\rho,z)$ coordinates by (\ref{orbit}),  of the form
\be
\td \hat{\star} (\rho J)=0,  \qquad J:= \Phi^{-1} \td \Phi  \; , 
\ee
where $\Phi$ is a real  symmetric positive-definite matrix, can be solved using the Mazur identity~\cite{Mazur:1984wz}. Note that the conformal factor in (\ref{orbit}) cancels in this PDE and therefore we can, and will, treat it as defined on  the $(\rho,z)$ half-plane with a fixed flat metric.  We recall the derivation of this identity.  

Thus suppose we have two solutions $\Phi$ and $\tilde{\Phi}$ to this PDE and define their `difference'
\be
\Psi= \tilde{\Phi} \Phi^{-1}- \text{I}  \; .
\ee
Elementary calculations reveal that $\td \Psi= \tilde{\Phi} \mathring{J} \Phi^{-1}$, 
where $\mathring{J} := \tilde{J}- J$,  and $J^T= \Phi J \Phi^{-1}$, 
where in the latter relation we have used symmetry of $\Phi$. The Mazur identity then follows:
\be
\td \hat{\star} (\rho \, \td \text{Tr} \Psi)= \rho \text{Tr} ( \mathring{J}^T \tilde{\Phi} \cdot \mathring{J}\Phi^{-1})   \; .
\ee
For this to be useful we need the r.h.s. to be positive definite. This indeed follows from the property that $\Phi$ is positive-definite, which allows one to write $\Phi= S S^T$ for some $S$ real invertible matrix (the square root matrix). Then we can write the identity as
\be
\td \hat{\star} (\rho \, \td \text{Tr} \Psi)= \rho \text{Tr} (N^T  \cdot N)   \; ,  \label{mazur}
\ee
where $N:= \tilde{S}^T \mathring{J} S^{T-1}$ so the r.h.s. is manifestly non-negative. It follows that integrating the  Mazur identity (\ref{mazur}) over $\hat{M}$  gives
\be
\int_{\partial \hat{M}} \rho\,  \partial_n  \text{Tr} \Psi \;  \td \text{vol} = \int_{\hat{M}}  \rho \;  \text{Tr} (N^T  \cdot N)  \td \text{vol} \geq 0 \; ,\label{intmazur}
\ee
where $n$ is the unit outward normal to the boundary.
Furthermore, the r.h.s vanishes iff $N=0$ which is equivalent to $\mathring{J}=0$ and from above this is equivalent to $\Psi$ being a constant; therefore  if $\Psi$ vanishes at a point, it vanishes everywhere establishing $\tilde{\Phi}=\Phi$ and hence uniqueness. Thus the uniqueness problem reduces to showing that the boundary integral on the l.h.s of (\ref{intmazur}) vanishes. \vskip.3cm

\begin{myproof}{Theorem}{\ref{th:unique}}
Consider (\ref{geq}) which takes the above form with $\Phi=g$ and the corresponding integrated Mazur identity (\ref{intmazur}). The integral over the boundary $\partial \hat{M}$ receives two contributions, one from the semi-circle at infinity and one from the axes.  We will prove that both of these give zero contribution and hence by the above remarks establish the uniqueness theorem.

First consider the contribution from the semi-circle at infinity. From our asymptotic analysis, summarised in Proposition \ref{prop:AF}, we find that the difference matrix is 
\be
\Psi= \tilde{g} g^{-1}- \text{I}= \begin{pmatrix} \frac{-2\tilde{m}+2m}{r}+ O_2(r^{-2}) & \frac{2\tilde{j}-2j}{r^3}+ O_2(r^{-4}) \\ \frac{(2\tilde{j}-2j)\sin^2\theta}{r}+O_2(r^{-2}) & \frac{-2m+2\tilde{m}}{r}+ O_2(r^{-2}) \end{pmatrix}  \; .
\ee
Therefore, since $n= \partial_r$ we get $ \partial_n \text{Tr} \Psi= O(r^{-3})$ and using $\text{dvol} = r \td \theta$ we deduce that the contribution from the boundary integral at infinity on the l.h.s. of (\ref{intmazur})  vanishes. 

Now consider the contribution from the axes $\mathcal{A}$ which reduces to
\be
\sum_{A=1}^{N+1} \int_{I_A} \rho \partial_\rho \text{Tr}\Psi  \td z  \; .
\ee
Suppose the two solutions $g, \tilde{g}$ have the same rod structure. Then,  since the two solutions have the same rod vectors we may parameterise $g, \tilde{g}$ near $I_A$ relative to the same adapted basis $(u_A, v_A)$ so that the Gram matrix takes the form (\ref{axisg})  and we distinguish the quantities $h, w$ using tildes in the obvious manner. Then  a computation reveals
\be
\text{Tr} \Psi= \frac{ (h- \tilde{h})^2+ \rho^2 (w- \tilde{w})^2}{h \tilde{h}}  \; ,  \label{TrPsi}
\ee
which is positive definite (note that $L$ and $L_A$ drops out). Now, since smoothness of the axis requires $h>0$ and $w$ to be smooth functions of $\rho^2$ at $\rho=0$ it immediately follows that $\rho \partial_\rho \text{Tr} \Psi=0$ on  every rod $I_A$.  Therefore the contribution from the axis vanishes.
\end{myproof}
\vskip.3cm

We remark that the above uniqueness theorem does not require one to fix the asymptotic invariants $m,j$ of the two solutions to be the same. This is in contrast to the black hole uniqueness theorem which requires one to fix the angular momentum.  Indeed, it is  a consequence of the above theorem that the asymptotic invariants $m,j$ are fixed in terms of the rod structure and the  periodicities $\beta, \Omega$ that define the asymptotic geometry. 

For example, if one considers the euclidean Kerr rod structure, the free parameters appearing in the uniqueness theorem are $\beta, \Omega$ and the rod length $\ell_2$ of $I_2$. Then the absence of a conical singularity at $I_2$ fixes one parameter leaving a 2-parameter family corresponding to the euclidean Kerr solution. Similarly, the Chen-Teo instanton is also a 2-parameter family. Generalising this argument we may conjecture the dimension of the moduli space of regular solutions. The above uniqueness theorem shows that a solution with $N-1$ finite rods and  given rod vectors can be specified by $\beta, \Omega, \ell_2, \dots, \ell_N$, which comprises $N+1$ moduli.  Then, we expect that removal of the potential conical singularities at the finite rods gives $N-1$ constraints on the parameters. Therefore, given any rod structure, {\it if} regular solutions exist we expect the moduli space to be 2-dimensional.   

Of course, the above does not address the existence question: for what rod structures do solutions actually exist?  We will turn to this question  next. 

\subsection{Harmonic map formulation}

We now show that (\ref{geq}) can be equivalently written as a harmonic map.
Let 
\be
\Phi := \rho^{-1} g
\ee
which is a real, symmetric, positive-definite  matrix with $\det \Phi=1$.  Then $\Phi^{-1} \td \Phi = g^{-1} \td g -\rho^{-1}\td \rho$
and therefore (\ref{geq}) implies
\be
\td \hat{\star}( \rho \Phi^{-1} \td \Phi)=0  
\ee
since $\rho$ is a harmonic function.  Observe that the Mazur difference $\tilde{\Phi}\Phi^{-1}-I = \tilde{g} g^{-1} -I$ so the proof of Theorem \ref{th:unique} given in the previous section works in an identical fashion in terms of $\Phi$.

It is convenient to rewrite the above equation as
\be
\nabla \cdot (\Phi^{-1} \nabla \Phi)=0 \; ,  \label{HM}
\ee
where $\nabla$ is the standard derivative on euclidean space $\mathbb{R}^3$ and $\Phi$ is an axisymmetric function on $\mathbb{R}^3 \backslash \Gamma$ (to see this note that $(\rho, z)$ can be identified with cylindrical coordinates). Here $\Gamma$ is the $z$-axis and must be removed since $\Phi$ is singular there (i.e. at $\rho=0$).  Thus $\Phi: \mathbb{R}^3\backslash \Gamma \to N$ is a harmonic map, where the target $N$ is given by the Riemannian symmetric space $H^2 \cong SL(2, \mathbb{R})/SO(2)$.  To see this, one can introduce the standard parameterisation 
\be
\Phi= \begin{pmatrix} X+ X^{-1} Y^2 & X^{-1} Y \\  X^{-1} Y & X^{-1} \end{pmatrix}  \; ,
\ee
for $X>0$, which is explicitly unimodular, so the target space metric
\begin{equation}
G= \frac{1}{2} \text{Tr} [\Phi^{-1} \td \Phi \cdot \Phi^{-1} \td \Phi] = \frac{ \td X^2 + \td Y^2}{X^2}
\end{equation} 
is the standard metric on the Poincar\'e half-plane model of hyperbolic space.

For later reference, the tension of any map $\Phi: \mathbb{R}^3\backslash \Gamma \to N$ is encoded by the matrix
\be
\tau(\Phi):= \nabla \cdot (\Phi^{-1} \nabla \Phi)
\ee
and hence $\Phi$ is harmonic iff $\tau(\Phi)=0$.  Its norm with respect to the target space metric is given by
\begin{equation}
|\tau|^2 = \frac{1}{2} \text{Tr} [ (\nabla \cdot (\Phi^{-1} \nabla \Phi)  )^2] \; .  \label{tensionsq}
\end{equation}
We wish to address the existence of solutions to (\ref{geq}). The above shows this is equivalent to the existence of a harmonic map (\ref{HM}) with prescribed boundary conditions on $\Gamma$. Fortunately, Weinstein has established the following existence result for harmonic maps with target spaces $N$ that in particular include $N= SL(2, \mathbb{R})/SO(2)$.

\begin{theorem}[Weinstein~\cite{Weinstein:2019zrh}] 
\label{th:weinstein} Given a  map $\Phi_0: \mathbb{R}^3\backslash \Gamma \to N$ with $|\tau(\Phi_0)|$  bounded that decays at infinity 
sufficiently fast (called a { model map}), there exists unique harmonic map $\Phi : \mathbb{R}^3\backslash \Gamma \to N$ such that $\text{dist}_{N}(\Phi, \Phi_0)$ is  bounded and decays to zero at infinity ($\Phi$ is said to be {asymptotic} to $\Phi_0$).
\end{theorem}

Therefore the existence problem reduces to exhibiting a model map with the required asymptotics. This is achieved as follows.

\begin{prop} \label{prop:modelmap}
There exists an AF model map which exhibits any admissible rod structure.
\end{prop}

\begin{proof}
Consider the general case with $N$ corners, so that there are $N+1$ rods $I_A$.  We will  divide the $(\rho, z)$ half-plane into several regions.
On the interior of each rod we define a transition region $T_A \subset  I_A$ and thicken them into regions $\mathcal{T}_A = T_A \times [0,\rho_0)$ for some $\rho_0>0$.  Next, define  $\mathcal{R}$ to be the region  $|z| < R$ and $0\leq \rho < \rho_0$.  Then $\mathcal{R} \setminus \mathcal{T}_A$ leaves $N+2$ regions $\mathcal{S}_A$ separated by the $N+1$ transition regions $\mathcal{T}_A$.  Note that for $A=2, \dots, {N+1}$ the regions $\mathcal{S}_A$ include  part of  the rods $I_{A-1}$ and $I_A$ and the corner $z_{A-1}$ that separates them, whereas $\mathcal{S}_1$ and $\mathcal{S}_{N+2}$ sits over $I_1$ and $I_{N+1}$ respectively.  Note also that the transition regions $\mathcal{T}_A$ do not contain any of the corner points $z_A$. We also define a region near infinity $C_R$ by  $\rho^2+z^2 > R^2$. See Figure \ref{fig:model} for a depiction of these regions.

\begin{figure}[h!]
\centering
{
\begin{tikzpicture}[scale=1.5, every node/.style={scale=0.6}]
\fill[ fill=black, opacity=0,very thick](-4,4)--(-4,0)--(4,0)--(4,4);
\draw[black,thick](-4.2,0)--(4.2,0)node[black,font=\large,right=.2cm]{\color{black}{$z$}};
\draw[black,thick](0,0)--(0,4.2)node[black,font=\large,above=.2cm]{\color{black}{$\rho$}};
\draw[black,fill=black] (-1.5,0) circle [radius=.07] node[black,below=.4cm]{$z_1$};
\draw[black] (-1.5,1.2) node[black,above=.4cm]{};
\draw[black] (-1.5,3.6) node[black,above=.4cm]{$C_R$};
\draw[black,fill=black] (-3.8,0) circle [radius=.03] node[black,below=.4cm]{$-R$};
\draw[black,fill=black] (3.8,0) circle [radius=.03] node[black,below=.4cm]{$R$};
\draw[gray,dashed](-3.8,1.2)--(3.8,1.2);
\draw[gray](-3.8,0)--(-3.8,1.2);
\draw[gray](3.8,0)--(3.8,1.2);
\draw[black] (-3.4,0) node[black,above=.8cm]{$\mathcal{S}_{1}$};
\draw[gray](-3,0)--(-3,1.2);
\draw[black] (-2.6,0) node[black,above=.8cm]{$\mathcal{T}_{1}$};
\draw[gray](-2.2,0)--(-2.2,1.2);
\draw[black] (-1.5,0) node[black,above=.8cm]{$\mathcal{S}_{2}$};
\draw[gray](-0.5,0)--(-0.5,1.2);
\draw[black] (-0.25,0) node[black,above=.8cm]{$\mathcal{T}_{2}$};
\draw[gray](0.5,0)--(0.5,1.2);
\draw[black] (1.35,0) node[black,above=.8cm]{$\mathcal{S}_{3}$};
\draw[black] (1.35,1.2) node[black,above=.4cm]{};
\draw[gray](2.2,0)--(2.2,1.2);
\draw[black] (2.6,0) node[black,above=.8cm]{$\mathcal{T}_{3}$};
\draw[gray](3,0)--(3,1.2);
\draw[black] (3.4,0) node[black,above=.8cm]{$\mathcal{S}_{4}$};
\draw[black](0,0) node[black,below=.4cm]{\color{black}{}};
\draw[black,fill=black] (1.7,0) circle [radius=.07] node[black,below=.4cm]{$z_{2}$};
\draw[black] (3.8,0) arc (0:180:3.8cm);
\end{tikzpicture}}
\caption{Regions for the model map in the case $N=2$.\label{fig:model} }
\end{figure}
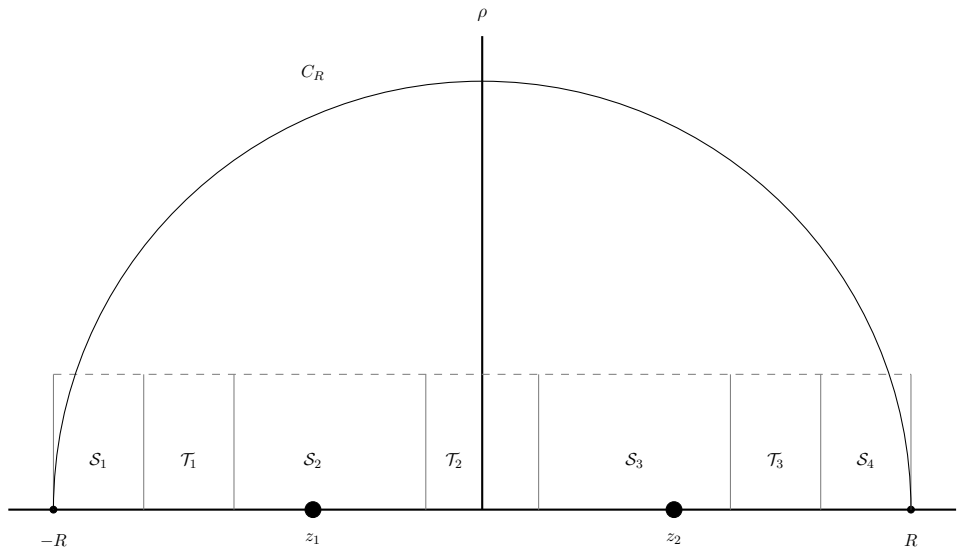
 We will define the model map $\Phi_0$ by specifying it on each of these regions. On the  remaining compact region $\rho^2+z^2\leq R^2$ and $\rho\geq \rho_0$ we take $\Phi_0$ to be any smooth extension so that this defines the model map everywhere.   It will be useful to define the following functions  
\be
\mu_A^\pm := \sqrt{\rho^2+ (z-z_A)^2} \mp (z-z_A)
\ee
which in particular have the property that $\mu_A^+=0$ for $\rho=0, z>z_A$ and $\mu_A^-=0$ for $\rho=0, z<z_A$.   It is also convenient to introduce the following constants
\be
\epsilon_{A-1}:= \det (v_{A-1}, v_A)   \label{epA}
\ee
for $A=2, \dots, N+1$ which by the admissibility condition (\ref{admissible}) must take unit values $\pm 1$.

We are now ready to define our model map.
It is convenient to work with $g$ instead of $\Phi$ and in the $2\pi$-periodic basis $(\partial_\psi, \partial_\chi)$ defined in (\ref{basis}). 
First, on $C_R$, $\mathcal{S}_1$ and $\mathcal{S}_{N+2}$, we take flat space
\be\label{modelmap:flat}
g= L \begin{pmatrix} 1 & 0 \\ 0 & \rho^2 \end{pmatrix} L^T
\ee
where $L$ is given by (\ref{L}) (recall  $\det g= \hat{\beta}^2 \rho^2$ in this basis).  Note that $\text{ker}(g)_{\rho=0}$ is spanned by $(0,1)$ as it must be.

On $\mathcal{S}_2$ take
\be
g= \begin{pmatrix} \hat{\beta} & \epsilon_1 v_2^2 \\ 0 & -\epsilon_1 v_2^1\end{pmatrix} \begin{pmatrix} \mu_1^+ & 0 \\ 0 & \mu_1^- \end{pmatrix} \begin{pmatrix} \hat{\beta} & 0 \\ \epsilon_1 v_2^2 & -\epsilon_1 v_2^1\end{pmatrix} 
\ee
where $\epsilon_1$ is defined by (\ref{epA}). Hence if $z<z_1$ the kernel of $g_{\rho=0}$ is spanned by $\underline{v}_1$ and if $z>z_1$ it is spanned by $\underline{v}_2$. Also note that $\det g = \hat{\beta}^2 \rho^2$ since the admissibility condition $\pm 1=\epsilon_1=- v_2^1$ gives  $\epsilon_1 v_2^1=-1$. On the transition region $\mathcal{T}_1$ which separates $\mathcal{S}_1$ and $\mathcal{S}_2$ we set
\be
g= \begin{pmatrix} \hat{\beta} & b(z) \\ 0 & 1 \end{pmatrix} \begin{pmatrix} (\mu_1^+)^{\lambda(z)} & 0 \\ 0 &  \rho^{2(1-\lambda(z))}(\mu_1^-)^{\lambda(z)} \end{pmatrix} \begin{pmatrix} \hat{\beta} & 0 \\ b(z) & 1\end{pmatrix} 
\ee
where $b(z)$ smoothly interpolates between $\hat{\beta} \Omega$ for $z\in {S}_1$ and $\epsilon_1 v_2^2$ for $z\in {S}_2$, whereas $\lambda(z)$   interpolates from $0$ to $1$ and $0\leq \lambda(z)\leq 1$.  Note the determinant is $\det g = \hat{\beta}^2 \rho^2$ and the kernel of $g_{\rho=0}$ on $\mathcal{T}_1$ is indeed spanned by $\underline{v}_1$.

On any other given region $\mathcal{S}_A$, $A=2, \dots, N+1$, take
\begin{equation}
g = \tilde{M}_{A-1} \begin{pmatrix} \mu_{A-1}^- & 0 \\  0 & \mu_{A-1}^+ \end{pmatrix} (\tilde{M}_{A-1})^T , \qquad \tilde{M}_{A-1} = \begin{pmatrix} v^2_{A}& -\epsilon_{A-1}\hat{\beta} v^2_{A-1} \\ -v^1_{A} &   \epsilon_{A-1} \hat{\beta} v^1_{A-1} \end{pmatrix}
\end{equation} 
which can be equivalently written as
\begin{equation}
g =  M_{A-1} \begin{pmatrix} \mu_{A-1}^+ & 0 \\  0 & \mu_{A-1}^- \end{pmatrix} (M_{A-1})^T , \qquad M_{A-1} = \begin{pmatrix}  \hat{\beta} v^2_{A-1} & \epsilon_{A-1} v^2_{A} \\ -\hat{\beta} v^1_{A-1}  &   -\epsilon_{A-1} v^1_{A} \end{pmatrix}
\end{equation} 
where $\epsilon_{A-1}$ is defined by (\ref{epA}) 
and note that $\tilde{M}_{A-1} = M_{A-1} C_{A-1}$ with 
\begin{equation}
C_{A-1} =\epsilon_{A-1} \begin{pmatrix} 0 & -1\\ 1 & 0 \end{pmatrix}
\end{equation} 
relates these two ways of writing the metric (which just swaps the two columns and scales appropriately). Note that on $\mathcal{S}_2$ this agrees with above. Observe that if $z \in {S}_A$, 
\begin{equation}
\text{ker}(g)_{\rho =0} = \begin{cases} \text{span}(\underline{v}_{A-1}) & z < z_{A-1} \\ \text{span}(\underline{v}_A) & z > z_{A-1} \end{cases}
\end{equation} so these indeed exhibit the required rod structure in these regions. This follows from
\be
M_{A-1}^{T} \begin{pmatrix} v^1_A \\ v^2_A \end{pmatrix}= -\hat{\beta} \epsilon_{A-1} \begin{pmatrix} 1 \\ 0 \end{pmatrix}  \; ,\qquad M_{A-1}^{T} \begin{pmatrix} v^1_{A-1} \\ v^2_{A-1} \end{pmatrix} =\begin{pmatrix} 0 \\ 1 \end{pmatrix}  \; .
\ee
 Observe that 
 \be
 \det M_A=\hat\beta
 \ee
  in view of the admissibility condition $\epsilon_{A-1}=\pm 1$ and hence $\det g =\hat{\beta}^2 \rho^2$.
  
On the transition regions $\mathcal{T}_A$, $A=2, \dots, N$, which separate $\mathcal{S}_A$ and $\mathcal{S}_{A+1}$, take
\begin{equation}\label{g_T}
g = N_A(z) \begin{pmatrix}  (\mu^-_{A-1})^{\lambda(z)} (\mu_A^+)^{1 - \lambda(z)} & 0 \\0 & (\mu^+_{A-1})^{\lambda(z)} (\mu_A^-)^{1 - \lambda(z)} \end{pmatrix} N_A (z)^T, 
\end{equation} where
\begin{equation}
N_A (z) = \begin{pmatrix} \gamma(z) v^2_A & \alpha^2_A(z) \\  -\gamma(z) v^1_A  & -\alpha^1_A(z)\end{pmatrix}
\end{equation} and we have defined smooth  functions $\alpha_A^i(z), \lambda(z)$ such that  $0\leq \lambda(z)\leq 1$ with $\lambda(z)$ varying smoothly from $1$ for $z \in {S}_A$ to $0$ for $z \in {S}_{A+1}$ and $\alpha^i_A(z)$ smoothly varies such that
\begin{equation}
\alpha^i_A(z) = \begin{cases}-\epsilon_{A-1} \hat{\beta}  v^i_{A-1} \qquad z \in {S}_A \\ \epsilon_A v^i_{A+1} \qquad z \in S_{A+1} \end{cases} 
\end{equation} 
and $\gamma(z)>0$ varies from $1$ for $z\in S_A$  to $\hat{\beta}$ for $z\in S_{A+1}$, and  $\det{N}_A \neq 0$ for $z \in {T}_A$.  For $z \in {T}_A$, we have
\be
g|_{\rho =0} = N_A(z)\begin{pmatrix} 4 |z-z_{A-1}|^{\lambda(z) }| z-z_A|^{1-\lambda(z)} & 0 \\ 0 & 0 \end{pmatrix} N_A(z)^{T}
\ee
and hence 
\begin{equation}
\text{ker} (g)_{\rho =0} = \text{span}(\underline{v}_A) ,
\end{equation} since 
\begin{equation}\label{N_A}
N_A(z)^T \begin{pmatrix} v_A^1 \\ v_A^2 \end{pmatrix} =\det(\underline{v}_A, \underline{\alpha}_A) \begin{pmatrix} 0 \\ 1 \end{pmatrix}.
\end{equation}
Therefore we have the correct rod structure on the transition region $T_A$.  Note that since $\det g= \hat{\beta}^2\rho^2$ we require $\det N_A=\pm \hat{\beta}$ which can be achieved by taking
\be\label{alpha1}
\alpha^i_A(z)=  \hat{\beta}\gamma(z)^{-1} \left( -\epsilon_{A-1} \lambda(z) v_{A-1}^i +\epsilon_A v_{A+1}^i (1-\lambda(z)) \right)
\ee
where  recall $\lambda(z)$ smoothly interpolates from $1$ on ${S}_A$ to $0$ on ${S}_{A+1}$. With this choice of $\alpha^i_{A}(z)$, 
\begin{equation}
\det N_A = \hat{\beta}( \lambda (\epsilon_{A-1})^2+(1-\lambda) (\epsilon_A)^2)= \hat{\beta}  \; ,
\end{equation} 
again using the admissibility condition, as claimed.

Finally, we must specify the final transition region $\mathcal{T}_{N+1}$ which separates $\mathcal{S}_{N+1}$ and $\mathcal{S}_{N+2}$. Recall on $\mathcal{S}_{N+2}$ we take the solution at infinity \eqref{modelmap:flat}.  On $\mathcal{S}_{N+1}$ the above gives
\be
g = \tilde{M}_{N} \begin{pmatrix} \mu_{N}^- & 0 \\  0 & \mu_{N}^+ \end{pmatrix} (\tilde{M}_{N})^T , \qquad \tilde{M}_{N} = \begin{pmatrix} 1 & -\epsilon_{N}\hat{\beta} v^2_{N} \\ 0 &    \hat{\beta}  \end{pmatrix}
\ee
where we have used $\underline{v}_{N+1}=(0,1)$ and $\epsilon_N= \det(\underline{v}_N, \underline{v}_{N+1})= v_N^1$ and $\epsilon_N=\pm1$. Thus on $\mathcal{T}_{N+1}$ we can take
\be
g = \begin{pmatrix} a(z) & b(z) \\ 0 & \hat{\beta} a(z)^{-1} \end{pmatrix}  \begin{pmatrix} (\mu^-_N)^{\lambda(z)} & 0 \\ 0  & \rho^{2(1-\lambda(z))} (\mu^+_N)^{\lambda(z)} \end{pmatrix}\begin{pmatrix} a(z) & 0 \\ b(z) & \hat{\beta} a(z)^{-1} \end{pmatrix}
\ee
where $0\leq \lambda(z)\leq 1$ goes from $1$ on $S_{N+1}$ to $0$ on $S_{N+2}$, whereas $(a(z), b(z))$ go from $(1, -\epsilon_{N}\hat{\beta} v^2_{N})$ on $S_{N+1}$ to $(\hat{\beta}, \hat{\beta} \Omega)$ on $S_{N+2}$ and $a(z)>0$. Note $\det g = \hat{\beta}^2 \rho^2$  and the kernel of $g_{\rho=0}$  is spanned by $\underline{v}_{N+1}$ as required.

The above defines a map that is manifestly AF  with any admissible rod structure. For this to be a model map we also need to verify that the tension is bounded. Specifically we require boundness of
\be
|\tau |^2 = \frac{1}{2}\text{Tr} [( \nabla\cdot (g^{-1}\nabla g))^2]
\ee
where we have used (\ref{tensionsq}) and $ \nabla\cdot (\Phi^{-1}\nabla \Phi) = \nabla\cdot (g^{-1}\nabla g)$.
First note that on the region $\rho^2+z^2\leq R^2$ and $\rho\geq \rho_0$, where we take the map to be any smooth extension, $| \tau |$ must be bounded by compactness. Next, observe that the map $g$ is a solution to (\ref{geq}) on the region $C_R, \mathcal{S}_1, \mathcal{S}_{N+2}$ (it is just flat space), and since $\Delta \log \mu_A^\pm=0$ where $\Delta:=\nabla \cdot \nabla$ is the Laplacian on $\mathbb{R}^3$, it is also a solution on all the regions $\mathcal{S}_{A}$, so the tension vanishes on these regions. Therefore, it remains to check boundedness on the transition regions $\mathcal{T}_A$.

In order to prove boundedness of $|\tau |^2$, it is clearly sufficient to show that  every component of the tension matrix $\tau= \nabla\cdot (g^{-1}\nabla g)$ is bounded.  On every transition region we have $g= N D N^T$ (dropping labels for clarity) where $D$ is diagonal, and a computation gives
 \bea
 \tau &=& \nabla (N^{-T} D^{-1}) \cdot (N^{-1} \nabla N) (D N^T) + (N^{-T} D^{-1})\nabla\cdot(N^{-1} \nabla N) DN^T  \nonumber \\ &&+  (N^{-T} D^{-1}) (N^{-1} \nabla N)\cdot \nabla (D N^T) + \nabla N^{-T} \cdot (D^{-1} \nabla D) N^T \nonumber \\ &&+ N^{-T} \nabla\cdot (D^{-1} \nabla D) N^T + N^{-T}  (D^{-1} \nabla D) \cdot \nabla N^T+ \nabla \cdot( N^{-T} \nabla N^T)  \; .
 \eea
 We now consider each of these seven terms.  The last (seventh) term is clearly bounded since $N$ and $N^{-1}$ are functions of only $z$ which are smooth everywhere.  For the fifth term note that since $D$ is diagonal
 \be
\nabla\cdot( D^{-1} \nabla D)=\Delta \log D
\ee
and on the generic transition regions $\mathcal{T}_A, A=2, \dots, N$,
\be
\log D_A = \begin{pmatrix} \lambda(z) \log \mu_{A-1}^- +(1-\lambda(z)) \log\mu_A^+ & 0 \\ 0&  \lambda(z) \log \mu_{A-1}^+ +(1-\lambda(z)) \log\mu_A^- \end{pmatrix}
\ee 
where (using $\Delta \log \mu_A^\pm=0$),
\be
\Delta  (\lambda(z) \log \mu_{A-1}^\mp +(1-\lambda(z)) \log\mu_A^\pm) = \lambda'' ( \log \mu_{A-1}^\mp -   \log\mu_A^\pm) + \lambda' \left( \frac{\partial_z \mu_{A-1}^\mp}{\mu_{A-1}^\mp} -  \frac{\partial_z \mu_{A}^\pm}{\mu_{A}^\pm}  \right)
\ee
are clearly bounded since $T_A \subset I_A$ are bounded away from the endpoints $z_{A-1}, z_A$.  Therefore, the 5th term is bounded.  The 4th and 6th terms are also bounded since $N$ only depends on $z$ and $D^{-1} \partial_z D= \partial_z \log D$ is bounded since 
\be
\partial_z (\lambda(z) \log \mu_{A-1}^\mp +(1-\lambda(z)) \log\mu_A^\pm) = \lambda' ( \log \mu_{A-1}^\mp -   \log\mu_A^\pm) + \lambda  \frac{\partial_z \mu_{A-1}^\mp}{\mu_{A-1}^\mp} +(1-\lambda)  \frac{\partial_z \mu_{A}^\pm}{\mu_{A}^\pm}  
\ee
is bounded from our previous comments. The 2nd term is also clearly bounded since in particular $D$ and $D^{-1}$ are on the transition regions (recall $\mu_A^\pm$ and $\mu_{A-1}^\pm$ are bounded away from zero on $\mathcal{T}_A$). Finally, the 1st and 3rd terms are bounded since  $N$ depends only on $z$ and $\partial_z (D N^T)$ and $\partial_z (D N^T)^{-1}$ are bounded since $\partial_z D, \partial_z D^{-1}$ are by the above comments.

It remains to check boundedness of the tension on the first and last transition regions $\mathcal{T}_1$ and $\mathcal{T}_{N+1}$. The above argument does not work. Instead, by performing an explicit calculation we find that
in the transition region $\mathcal{T}_1$ one gets, in the limit $\rho \to 0$,
\begin{equation}
\begin{aligned}
\tau^1_{~1} &= -\tau^2_{~2} = (z-z_1)^{-1} \left[2 \lambda' + \lambda'' (z-z_1) \log(2(z_1 - z))\right]+O(\rho^2), \qquad \tau^1_{~2} = O(\rho^2) \\
\tau^2_{~1} & = \frac{1}{z-z_1} \left[-2\lambda b' - 4 b \lambda'\right] - 2 \log(2(z_1-z))( b \lambda')' + b''  +O(\rho^2) \; ,
\end{aligned}
\end{equation} 
which  shows that the tension is indeed bounded near the axis and hence must be bounded on the whole of $\mathcal{T}_1$. A similar calculation shows that the tension is also bounded on $\mathcal{T}_{N+1}$.
\end{proof}

 The final part of the argument requires showing that the harmonic map is also AF and exhibits the same (arbitrary) rod structure as the model map.  This is established by the following result.  

\begin{prop}
\label{prop:asymp}
If a harmonic map  $\Phi$ is asymptotic to the model map $\Phi_0$ defined by Proposition \ref{prop:modelmap},  then it has the same rod structure as $\Phi_0$ and is AF.
\end{prop}

\begin{proof}
We adapt a proof given in the context of five-dimensional stationary and biaxisymmetric black hole spacetimes~\cite{Khuri:2017xsc}. This begins by noting that uniform boundedness of $\text{dist}_N(\Phi, \Phi_0)$ is equivalent to uniform boundedness of the Mazur difference $\text{Tr}(\Phi \Phi_0^{-1}-I)$. Furthermore, it can also be shown that $\text{dist}_N(\Phi, \Phi_0) \to 0$ is equivalent to $\text{Tr} ( \Phi \Phi_0^{-1}- I) \to 0$. 

  In our case $\text{Tr}(\Phi \Phi_0^{-1})= \text{Tr}(g g_0^{-1})$, so we may work directly with the Gram matrices. Since $g_0$ is a positive-definite symmetric matrix we can diagonalise it by an orthogonal matrix $O$ and write $g_0= O D O^T$ where $D= \text{diag}(\lambda_1, \lambda_2)$.  Then
\be
\text{Tr}( g g_0^{-1}) = \text{Tr}( \tilde{g} D^{-1})= \frac{\tilde{g}_{11}}{\lambda_1}+ \frac{\tilde{g}_{22}}{\lambda_2}
\ee
where $g= O \tilde{g} O^T$.  Now, on the interior of any rod $I_A$ we have $\text{rank}(g_0)_{\rho=0}=1$ so without loss of generality we can assume $\lambda_1|_{\rho=0}=0$ and $\lambda_2|_{\rho=0}>0$. Now noting that $\rho^2= \det g= \lambda_1\lambda_2$, we deduce that boundedness $\text{Tr}( g g_0^{-1})\leq c$ implies $\tilde{g}_{11}\leq c \lambda_1 = c \rho^2 \lambda_2^{-1}$ and $\tilde{g}_{22}\leq c \lambda_2$.   Furthermore, $\rho^2= \det \tilde{g}$ implies that $\rho^2\leq \tilde{g}_{11} \tilde{g}_{22}$ and hence $\tilde{g}_{22}\geq c^{-1} \lambda_2$, as well as $\tilde{g}_{12}^2 \leq \tilde{g}_{11} \tilde{g}_{22} $ and hence $\tilde{g}_{12}^2\leq c^2 \rho^2$.    Therefore, it follows that on the axis $\rho=0$ we have
\be
\tilde{g}_{\rho=0}= \begin{pmatrix} 0 &0 \\ 0 & \tilde{g}_{22} \end{pmatrix} \; ,
\ee
where $\tilde{g}_{22}>0$. Thus, $\text{ker}(\tilde{g})=\text{ker}(D)$ and hence it follows that $g$ and $g_0$ have the same kernel at $\rho=0$.  Thus, $g$ and ${g}_0$  have identical rod structure.

To establish AF it is convenient to use the following parameterisation for the Gram matrix with respect to the basis $(\partial_\tau, \partial_\phi)$ ,
\be
g = \begin{pmatrix}  h+ \rho^2 h^{-1} w^2 & \rho^2 h^{-1} w \\ \rho^2 h^{-1} w & \rho^2 h^{-1}  \end{pmatrix} 
\ee
and similarly for the model map (this is just like (\ref{axisg}) with the factors of $L$ removed).  
This gives
\be
\text{Tr} ( g g_0^{-1}- I) = \frac{(h-h_0)^2+ \rho^2 (w-w_0)^2}{h h_0} \; .
\ee
In fact, in this parameterisation our model map for $r>R$ is simply $h_0=1$ and $w_0=0$.  Therefore, since $\text{Tr} (g g_0^{-1}- I) \to 0$ as $r \to \infty$, we deduce that $h\to 1$ and $\rho w \to 0$. This  shows $g_{\tau\tau}\to 1$ and $g_{\phi\phi}\to r^2 \sin^2\theta$.  It then follows from $\det g= \rho^2$ that $g_{\tau \phi}\to 0$ and hence  is asymptotic to (\ref{g0}).
\end{proof}

\begin{remark}
In fact, to prove that the harmonic map is AF as in our Definition \ref{def:AF}, we must also prove decay estimates for the subleading terms. We will not attempt this here, although presumably this could be achieved by a similar method to that used for stationary and axisymmetric metrics~\cite{wein2}.
\end{remark}

We are now ready to establish our final result.  \vskip.3cm

\begin{myproof}{Theorem}{\ref{th:exist}}
By Theorem \ref{th:weinstein}  it follows that there exists a unique harmonic map that is asymptotic to our model map constructed in Proposition \ref{prop:modelmap}. Then apply Proposition \ref{prop:asymp}.
\end{myproof}

\begin{remark}
As a simple application of Theorem \ref{th:exist} we may completely classify {\it static} instantons, i.e., under the further assumption that $\partial_\tau$ is hypersurface orthogonal (equivalently, its orthogonal distribution is involutive). In the parameterisation (\ref{Vom}) this is equivalent to $\omega$ being a constant, which by asymptotic flatness must vanish.  By considering a general rod vector in the $(\partial_\psi, \partial_\chi)$ basis, one can show that the orthogonality condition $g(\partial_\tau, \partial_\phi)=0$, along with the admissibility condition,  implies that the only possible rod vectors are $(0,1)$ or $(1, - \hat{\beta}\Omega)$ where $\hat{\beta} \Omega \in \mathbb{Z}$. In particular, this implies $\hat{\beta}\partial_\tau$ has $2\pi$-periodic orbits and therefore without loss of generality we may set $\Omega=0$. Therefore the general rod structure is given by  rod vectors $(0,1)$ and $(1,0)$.  By our existence theorem there exists a unique solution for any rod structure with such rod vectors parameterised by the finite rod lengths $\ell_A$.  In fact, it is easy to write down the explicit metric in this case, which is the euclidean multi-Schwarzschild solution.  However, it was shown that if $N>2$ such a metric is necessarily conically singular on the axis~\cite{Gibbons:1979nf}. Therefore, the only regular static AF toric instanton is the  euclidean Schwarzschild metric (\ref{Sch}) on $\mathbb{R}^2 \times S^2$. 
\end{remark}

\section{General identities}
\label{sec:id}

In this section we will collect various general identities that we have obtained for AF toric gravitational instantons.

\subsection{Thermodynamic identities}

In this section we derive a general set of identities for AF toric instantons that relate the asymptotic invariants $m, j$ to certain local geometric charges. The latter are defined as follows.
 Each finite rod $I_A$ corresponds to a non-contractible 2-surface $C_A \cong S^2$ in $(M, \mathbf{g})$ and we may define charges associated to each 2-cycle $[C_A]$ by
\be
Q_k [C_A] := \int_{C_A} \star \td k  \; ,   \label{topcharge}
\ee
where $k$ is a Killing field of $(M, \mathbf{g})$.
The integrand is closed by the fact that $(M,\mathbf{g})$ is Ricci flat and hence the integrals can be taken over any surfaces homologous to $C_A$, i.e. any representative of $[C_A]$. Thus these are topological charges.  Before stating our main result it is useful to note the following.

\begin{lemma} For every finite axis rod $I_A$, 
\be
 Q_{v_A}[C_A]= 2{\beta} \ell_A  \; ,  \label{lA}
\ee
where $v_A$ is the rod vector of $I_A$ corresponding to the 2-cycle $[C_A]$.
\end{lemma}

\begin{proof}
We need to evaluate 
\be
Q_{v_A}[C_A] =\int_{C_A} \star \td v_A  \; .
\ee
It is convenient to use a $2\pi$-periodic oriented-basis $(u_A, v_A)$ adapted to $I_A$, so we can write the Gram matrix  as (\ref{axisg})
where $h>0$ and $w$ are smooth at $\rho=0$.   
Then by direct calculation we find
\be
(\star \td v_A) \to  2\hat{\beta}  \td z  \wedge \td \phi^1
\ee
as $\rho \to 0$ and (\ref{lA}) immediately follows by integration upon use of $\int_{C_A} \alpha = - 2\pi \int_{I_A} \iota_{u_A} \alpha$ where $\alpha$ is a $U(1)^2$-invariant two-form.
\end{proof}

We are now ready to state our main result.
\begin{theorem}
\label{th:id}
Let $(M, \mathbf{g})$ be an AF toric gravitational instanton with rod structure $\mathcal{R}$ and asymptotic invariants $(m ,j)$. Then
\bea
&&j=  - \frac{1}{16 \pi} \sum_{C_A}  Q_{\partial_\chi} [C_A] v^1_A  \; ,  \label{id2} \\
&&m-2 \Omega j= \frac{1}{2\beta} \sum_{C_A}\left( A[C_A]- \frac{1}{2}Q_{\partial_\chi} [C_A] v^2_A \right) \; , \label{id3}
\eea
where $A[C_A]$ is the area of $C_A$ defined by (\ref{area}), $Q_{\partial_\chi}$ is the charge (\ref{topcharge}) and $v^i_A$ are the rod vectors relative to the basis (\ref{basis}).
\end{theorem}

\begin{proof}
We will work in the $2\pi$-periodic basis $(\partial_\psi, \partial_\chi)$.
From the asymptotics given in Proposition \ref{prop:AF} we have
\be
g = L \left( \begin{array}{cc}  1-\frac{2m}{r}+O_2(r^{-2}) & \frac{2j \sin^2\theta}{r}+O_2(r^{-2})  \\  \frac{2j \sin^2\theta}{r} +O_2(r^{-2}) & r^2 \sin^2\theta \left(1+ \frac{2m}{r} + O_2(r^{-2}) \right) \end{array} \right) L^T
\ee
where  $L$ is given by $(\ref{L})$ and note that in this basis $\det g=\hat{\beta}^2 \rho^2$.
A computation reveals that
\be
\rho  g^{-1} \partial_r g  = L^{-T}\left( \begin{array}{cc}  \frac{2m\sin \theta}{r}+O(r^{-2})  & - \frac{6j \sin^3\theta}{ r}+O(r^{-2})  \\ O(r^{-3}) & 2 \sin \theta (1- \frac{m}{ r} )+O(r^{-2})  \end{array} \right) L^T
\ee
and
\be
\rho g^{-1} \partial_\theta g = L^{-T} \left( \begin{array}{cc}  O(r^{-3})  & O(r^{-2}) \\ O(r^{-2}) & 2r \cos \theta + O(1)  \end{array} \right)L^T  \; .
\ee
From (\ref{geq})  we know that 
\be
\Theta :=\hat{\beta} \rho g^{-1} \star_2 \td g + \td C
\ee
 is a closed 1-form where $C$ is an arbitrary matrix of functions (the factor of $\hat{\beta}$ is introduced  for convenience). Then we may choose $C$ to eliminate the divergent terms in the asymptotics. Choose the orientation $\star_2 \td r = r \td \theta$ and $\star_2 \td \theta = - \td r /r$ (so that $\td \rho \wedge \td z$ has positive orientation). We find that a simple choice is
\be
C = \hat{\beta}L^{-T} \left( \begin{array}{cc} 0 & 0 \\ 0 & 2 z \end{array} \right) L^T = \left( \begin{array}{cc} 0 & 0 \\ 2\hat{\beta}^2 \Omega z & 2\hat{\beta} z \end{array} \right) 
\ee
which gives
\be
\Theta =  \hat{\beta} L^{-T} \left( \begin{array}{cc}  2m\sin \theta+ O(r^{-1})  & - 6 j\sin^3\theta +O(r^{-1})   \\  O(r^{-2}) & -2m \sin \theta +O(r^{-1})  \end{array} \right) L^T  \td\theta  + O(r^{-2}) \td r \; .
\ee
Therefore integrating this over the semi-circle at infinity from $\theta=\pi$ to $\theta=0$ (i.e. clockwise in the $z\rho$ plane since this is the induced orientation) gives
\be
\int_\infty \Theta 
= -\hat{\beta} \begin{pmatrix} 4(m-2 \Omega j) & - 8 j \hat{\beta}
^{-1} \\ 8\hat{\beta}\Omega ( -m+ j \Omega) & -4 (m - 2\Omega j)\end{pmatrix}  \; .  \label{intThetainf}
\ee
On the other hand, integrating $\td \Theta=0$ over $\hat{M}$ gives
\be
\int_\infty \Theta -\sum_{A=1}^{N+1} \int_{I_A} \Theta=0  \; ,  \label{ID}
\ee
where the integrals over the $I_A$ are taken with a positive orientation along the $z$-axis (the minus sign arises as we are traversing the contour clockwise in the $z\rho$ plane).
These identities relate the integral at infinity (\ref{intThetainf}) to the axis integrals which we evaluate next.

The axis integrals can be evaluated in terms of certain potentials $Y_{ij}$ defined  by
\be
\td Y_{ij} := \star (\eta_i \wedge \td \eta_j)
\ee
where $\eta_i, i=1,2$ are a basis of the $U(1)^2$-Killing fields and the r.h.s. is a closed 1-form for each $i,j$ by Ricci flatness of $(M,\mathbf{g}).$\footnote{These are Riemannian versions of the twist and Ernst potentials used in General Relativity.} The existence of globally defined functions $Y_{ij}$ then follows since we assume $M$ is simply connected. A short computation reveals that
\be\label{dY}
\td Y_{ij} =- \eta_{ik}  \hat{\star} (\hat{\beta} \rho  g^{-1} \td g)^{k}_{~j}
\ee
where $\eta_{ij}$ is an antisymmetric symbol and $\eta_{12}=1$. To obtain this we used that the volume form of our 4-manifold has nonzero components determined by $\epsilon_{ij ab}=\hat{\beta} \rho \eta_{ij} \hat{\epsilon}_{ab}$, where the factor of $\hat{\beta}$ arises as we are in a basis where $\sqrt{\det g}=\hat{\beta} \rho$. Therefore we can write
\be
\Theta^i_{~j}= \td ( -\eta^{ik} Y_{kj}+ C^i_{~j})
\ee
where $\eta^{ij}$ is the inverse matrix of $\eta_{ij}$, so the finite axis rod integrals become (in matrix notation)
\be
\int_{I_A} \Theta =- \eta^{-1} [Y(z_{A})- Y(z_{A-1})] + 2 \left( \begin{array}{cc} 0 & 0 \\ \hat{\beta}\Omega & 1\end{array} \right) \hat{\beta} \ell_{A}  \; .
\ee
The semi-infinite axis integrals require special attention, which we turn to next. 

Near $I_1$ and $I_{N+1}$ we can parameterise the metric in the basis $(\partial_\psi, \partial_\chi)$ by (\ref{axisg}) where $h>0$ and $w$ are smooth at $\rho=0$.
Then, we find
\be
\Theta_z =\hat{\beta} L^{-T} \begin{pmatrix} O(\rho) & O(\rho^3) \\ 2 w + O(\rho) & O(\rho) \end{pmatrix} L^T = \begin{pmatrix} O(\rho) & O(\rho^3) \\ 2\hat{\beta}^2 w + O(\rho) & O(\rho) \end{pmatrix} 
\ee
and hence
\be
\int_{I_1} \Theta= \begin{pmatrix} 0 & 0 \\  \int_{I_1} 2 \hat{\beta}^2 w  & 0 \end{pmatrix} , \qquad \int_{I_{N+1}} \Theta= \begin{pmatrix} 0 & 0 \\   \int_{I_{N+1}} 2 \hat{\beta}^2 w & 0 \end{pmatrix}  \; .
\ee

We can express the finite axis integrals in terms of the topological charges (\ref{topcharge}) as follows.  We can reduce the integrals to the orbit space
\bea
Q_{\eta_i}[C_A] &=&- 2\pi \det L_A \int_{I_A} \iota_{u_A} \star \td \eta_i = 2\pi \det L_A \int_{I_A} \star( u_A \wedge \td \eta_i) \nonumber \\  &=& 2\pi \det L_A L_{A 1j} \int_{I_A} \star (\eta_j \wedge \td \eta_i)  \; ,
\eea
where recall $u_A$ is the nonvanishing $U(1)$ generator on $I_A$. Clearly, it is also true that 
\be
0= \int_{I_A} \star( v_A \wedge \td \eta_i)= L_{A 2j} \int_{I_A} \star (\eta_j \wedge \td \eta_i)  \; ,
\ee
since $v_A$ vanishes on $I_A$. We can write these as a matrix  equation
\be\label{Q}
\begin{pmatrix} Q_{\eta_1}[C_A] & Q_{\eta_2}[C_A] \\ 0 & 0 \end{pmatrix}  =2\pi \det L_A L_A [Y(z_{A})- Y(z_{A-1})] 
\ee
where we have used the definition of $Y_{ij}$.  Therefore,
\bea
\int_{I_A} \Theta & =&-(2\pi \det L_A)^{-1} \eta^{-1} L_A^{-1}\begin{pmatrix} Q_{\eta_1}[C_A] & Q_{\eta_2}[C_A] \\ 0 & 0 \end{pmatrix} + 2 \left( \begin{array}{cc} 0 & 0 \\ \hat{\beta}\Omega & 1\end{array} \right) 
 \hat{\beta} \ell_{A}   \; , \nonumber \\
& =& -\frac{1}{2\pi}\begin{pmatrix} Q_{\eta_1}[C_A] v^1_A& Q_{\eta_2}[C_A]v_A^1 \\ Q_{\eta_1}[C_A] v_A^2 & Q_{\eta_2}[C_A] v_A^2 \end{pmatrix} + 2 \left( \begin{array}{cc} 0 & 0 \\ \hat{\beta}\Omega & 1\end{array} \right)\hat{\beta} \ell_{A}   \; . 
\eea
Note that the $u_A$ vectors have dropped out of this expression which is therefore manifestly invariant under the freedom  in the definition of $u_A$, given by $u_A\to u_A+ n v_A$ where $n \in \mathbb{Z}$, as it must be.

Putting everything together we arrive at the identities
\bea
m-2 \Omega j = \frac{1}{4{\beta}}  \sum_{C_A}  Q_{\partial_\psi}[C_A] v^1_A  \; , \label{id1}
\eea
and (\ref{id2}) and (\ref{id3}),  together with an identity for $\int_{I_1 \cup I_{N+1} }w$, where we have used (\ref{area}) and the fact that our  basis is $(\partial_\psi, \partial_\chi)$. In fact, (\ref{id1}) and (\ref{id3}) are equivalent upon use of the identities (\ref{lA}).
Therefore, the content of our identities reduces to (\ref{id2}) and (\ref{id3}) together with (\ref{lA}).
\end{proof}

\begin{remark}
Observe that, combining (\ref{id2}) and (\ref{id3}) we can extract an identity for $m$:
\be
m = \frac{1}{2\beta}  \sum_{C_A} \left(A[C_A]- \frac{1}{2}Q_{\partial_\chi}[C_A] (v^2_A+ \Omega \hat{\beta} v_A^1) \right)  \; .   \label{mass}
\ee
Thus we obtain a formula for the mass as a sum over invariants of the bolts $C_A$.
\end{remark}

Let us consider the special case of Kerr rod structure, so  $N=2$ and $\underline{v}_2=(1,0)$. 
Then (\ref{id3}) reduces to
\be\label{Kerr1}
m-2\Omega j= \frac{1}{2 \beta} A[C_2]
\ee
which is the euclidean version of the Smarr relation, whereas (\ref{id2}) gives
\be\label{Kerr2}
j=  -\tfrac{1}{16\pi}Q_{\partial_\chi} [C_2]
\ee
and the remaining independent identity is simply
\be\label{Kerr3}
A[C_2]= \frac{1}{2}Q_{\partial_\psi} [C_2]  \; .
\ee
These relations for $Q_{\eta_i}[C_2]$ are euclidean versions of the Komar integrals for the angular momentum and the horizon mass. We verify these identities for the explicit Kerr instanton in Appendix \ref{app:explicit}.

For the Chen-Teo instanton rod structure $N=3$, $\underline{v}_2=(1, 0)$, $\underline{v}_3=(1, -1)$, so we obtain
\bea
&&m- 2 \Omega j  = \frac{1}{2\beta}( A[C_2] +A[C_3]) + \frac{1}{4 {\beta}}Q_{\partial_\chi}[C_3]  \label{CTm} \\ 
&&j =   -\tfrac{1}{16 \pi}(Q_{\partial_\chi}[C_2]+Q_{\partial_\chi}[C_3])  \; . \label{CTj}
 \eea
We have checked these are indeed satisfied by the Chen-Teo instanton in  Appendix \ref{app:explicit}.

In general, we thus obtain euclidean versions of the well-known Smarr relation and Komar angular momentum and mass formulae for multi-black holes.

\subsection{Euclidean action}

The action for an AF instanton $(M, \mathbf{g})$ is given by the Einstein-Hilbert action together with the Gibbons-Hawking-York boundary term:  
\be
I = \frac{1}{16\pi} \int_M  R \sqrt{\mathbf{g}} \,\td^4 x -  \lim_{r\to \infty} \frac{1}{8\pi}\int_{\partial M_r} (K-K_0) \sqrt{h}\,  \td^3 x  \; ,
\ee 
where the boundary $\partial M_r$ is taken at fixed large $r$ in the asymptotic end in Definition \ref{def:AF},  $h$ the induced metric, $K$ the mean curvature in $(M, \mathbf{g})$ and $K_0$ is the mean curvature  in $(M_{\flat},\mathbf{g}_0)$~\cite{Gibbons:1976ue}.  This can be evaluated in exactly the same manner as for euclidean Kerr.  Indeed, since it depends only on the leading asymptotics given in Proposition \ref{prop:AF}, and this is the same as for the Kerr solution, we find that it takes exactly the same form,
\be 
I =\frac{\beta m}{2}  \; .
\ee
Upon combining this with (\ref{mass}) we deduce
\be
I= \sum_{C_A} \left( \frac{1}{4}A[C_A]- \frac{1}{8}Q_{\partial_\chi}[C_A] (v^2_A+ \Omega \hat{\beta} v_A^1) \right)  \; .
\ee
In the context of the euclidean approach to quantum gravity the euclidean action is minus the logarithm of the partition function~\cite{Gibbons:1976ue}. In order to compute the entropy one needs the $\beta$ dependence of the partition function, which unfortunately we do not have access to since regularity at the axes will impose nontrivial relations between the moduli $\ell_A, \beta, \Omega$.
Nevertheless, it is tantalising that a $1/4$ of the area of the bolts appears in the action.
A similar formula was found for compact gravitational instantons~\cite{Gibbons:1979xm}.

\appendix

\section{Curvature calculations} 
\label{app:curvature}

We require the Riemann curvature tensor of the metric
\be
\mathbf{g}= g_{ij}\td \phi^i \td \phi^j+ \hat{g}_{ab} \td x^a \td x^b
\ee
where  $\partial_{\phi^i}$ are Killing fields. In the coordinate chart $x^A=(\phi^i, x^a)$ we find that the non-vanishing components of the Christoffel symbols are given by
\be
\Gamma^a_{bc}= \hat{\Gamma}^a_{bc}, \qquad \Gamma^a_{ij}= - \tfrac{1}{2} \hat{g}^{ab}\partial_b g_{ij}, \qquad \Gamma^i_{j a}= \tfrac{1}{2} (J_a)^i_{~j}  \; ,
\ee
where recall $J$ is defined by (\ref{defs}).
The components of the Riemann tensor of ${g}_{AB}$ are 
\begin{equation}
R^A_{~BCD} = \partial_C \Gamma^A_{~BD} - \partial_D \Gamma^A_{~BC} + \Gamma^A_{~FC}\Gamma^F_{~BD} - \Gamma^A_{~FD}\Gamma^F_{~BC}
\end{equation}
and we find the non-vanishing components are given by 
\be
\begin{aligned}
R^a_{~bcd} &= \hat{R}^a_{~bcd} \\
R^i_{~jab} & = -\frac{1}{2} (J_{[a})^i_{~k} (J_{b]})^k_{~j}  \\
R^i_{~ajb} &= - \hat{\nabla}_b (J_a)^i_{~j}- \frac{1}{4} (J_b)^i_{~k} (J_a)^k_{~j} \\
R^{kl}_{\phantom{kl}ij}  & = -\frac{1}{2} (J_a)^k_{[i} (J^a)^l_{j]}  \; ,
\end{aligned}
\ee
where $\hat{R}^a_{~bcd}$ and $\hat{\nabla}$ are the Riemann tensor and metric connection of $\hat{g}$, and we have used (\ref{zerocurvature}) to simplify $R^i_{~jab}$. The Ricci tensor $R_{BD}= R^A_{~BAD}$ then has non-vanishing components
\be
\begin{aligned}
R_{ij} &= -\frac{1}{2\rho}g_{ik} \hat{\nabla}^a( \rho (J_a)^k_{~j}) \\
R_{ab} & = \hat{R}_{ab} - \hat{\nabla}_a \hat{\nabla}_b \log \rho - \frac{1}{4} \text{Tr}( J_a J_b)  \; ,
\end{aligned}
\ee
where recall $\rho$ is defined by (\ref{defs}), and we have simplified $R_{ij}$ using the identity 
\be
\text{Tr} \, J= \frac{2 \td \rho}{\rho}  \; .
\ee
It therefore immediately follows that Ricci-flatness of $\mathbf{g}$ is equivalent to (\ref{geq}) and (\ref{2dRic}).

\section{Explicit instantons}
\label{app:explicit}

\subsection{Euclidean Kerr instanton}
The Euclidean Kerr instanton is  a family of Ricci flat metrics on $M = \mathbb{R}^2 \times S^2$, 
\begin{equation}
\mathbf{g}_{\text{\tiny K}} = \frac{f}{\Sigma}\left(\td \tau + \alpha \sin^2\theta \td \phi\right)^2 + \Sigma\left(\frac{\td r^2}{f} + \td \theta^2\right) + \frac{\sin^2\theta}{\Sigma} \left((r^2 - \alpha^2) \td \phi- \alpha \td \tau\right)^2 \; ,
\end{equation}  
where $f := r^2 - 2 \mu r - \alpha^2$ and $\Sigma := r^2 - \alpha^2 \cos^2\theta$, which is parameterised by two constants $\mu>0, \alpha \geq 0$. The `radial' coordinate $ r \in (r_+, \infty)$ where $r_+ := \mu + \sqrt{\mu^2 + \alpha^2}$ is the positive root of $f(r)$ and $\theta \in (0,\pi)$. $(M, \mathbf{g}_{\text{\tiny K}})$ admits a torus action as isometries generated by $\frac{\partial}{\partial \tau}, \frac{\partial}{\partial \phi}$.  
This metric  is AF as $r\to \infty$ with $(\tau, \phi)$ coinciding with the canonical coordinates in Definition \ref{def:AF}  (note the $(r,\theta)$ coordinates here are different to those in Proposition \ref{prop:AF}). The asymptotic invariants can be easily computed and are found to be
\be
m=\mu, \qquad  j = -\alpha \mu  \; .
\ee
Clearly the subset  of metrics with $\alpha =0$ is the Schwarzschild instanton. 

We identify the canonical orbit space coordinates $(\rho,z)$ by
\begin{equation}
\rho = \sqrt{\det{g}} = \sqrt{f} \sin\theta, \qquad z = (r - \mu) \cos\theta.
\end{equation} Inspection of the zero set of $\rho$ reveals the rod structure: 
\begin{enumerate}
\item $I_1$: a semi-infinite rod $(-\infty, z_1)$ with $z_1 = -(r_+ - \mu)$ (corresponding to $r> r_+, \theta =\pi$); 
\item $I_2$: a finite rod $(z_1, z_2)$ with $z_2 = r_+ - \mu$ (corresponding to the set $r = r_+, 0 < \theta < \pi$);  
\item  $I_3$: a semi-infinite rod $(z_2, \infty)$ (corresponding to $r > r_+, \theta =0$)
\end{enumerate}  On $I_1 \cup I_3$ the vector $\partial_\phi$ vanishes, whereas on $I_2$ it is easy to verify that $\partial_\psi$ vanishes, where $\partial_\psi$  defined by (\ref{basis})
with
\begin{equation}
\Omega = \frac{\alpha}{r_+^2 - \alpha^2} , \qquad \beta = \frac{4\pi r_+ (r_+^2 - \alpha^2)}{r_+^2 + \alpha^2} \; ,
\end{equation}  
 degenerates smoothly and generates $2\pi$-periodic flows.  Thus in our canonical basis $(\partial_\psi, \partial_\chi)$  the rod vectors are $\underline{v}_1 = (0,1),  \underline{v}_2 =(1,0), \underline{v}_3 = (0,1)$. 

A straightforward computation using \eqref{dY} yields (with an appropriate choice of integration constant) 
\begin{equation}\begin{aligned}
Y_{\chi\chi} &= \frac{4a (r_+^2 - a^2)(2 r^2 + (r^2 + a^2)\cos\theta)\sin^4 \left(\frac{\theta}{2}\right)}{r_+ \Sigma}
 \\
Y_{\chi\psi} & = - \frac{2(r_+^2 - a^2)(2 a^2(r^2 - r_+^2)\cos\theta + a^2(r^2 +a^2) \cos 2\theta + a^4 - 2 r^2 r_+^2 - a^2 r^2)\sin^2 \left(\frac{\theta}{2}\right)}{(r_+^2 + a^2) \Sigma}
\\
 Y_{\psi\psi} &= \frac{2a r_+(r-r_+)^2 (r_+^2 - a^2) (4 r r_+ + 2r_+^2 + a^2 + a^2\cos 2\theta)\cos\theta}{(r_+^2 + a^2)^2 \Sigma} \\
 Y_{\psi\chi} & = \frac{(r- r_+)(r_+^2 - a^2)( 4r^2 r_+ - a^2 r - 3 a^2 r_+ + (r- r_+) a^2 \cos 2\theta)\cos\theta}{(r_+^2 + a^2) \Sigma}
 \end{aligned}
\end{equation} In particular we find that in the $(\partial_\psi,\partial_\chi)$ basis,
\begin{equation}
Y(z_2) - Y(z_1) = -\frac{\beta}{2\pi} \cdot \frac{2(r_+^2 + \alpha^2)}{r_+} \begin{pmatrix}  0 & 0 \\ 1 & \frac{\alpha}{r_+} \end{pmatrix}.
\end{equation}  For the change of basis matrix we take 
\begin{equation}
L_2 = \begin{pmatrix} 0 & 1 \\ 1 & 0 \end{pmatrix} , 
\end{equation} and therefore using \eqref{Q} we find
\begin{equation}
Q_{\partial_\psi}[C_2] = \frac{2\beta (r_+^2 + \alpha^2)}{r_+}, \qquad Q_{\partial_\chi}[C_2] = \frac{2\beta \alpha (r_+^2 + \alpha^2)}{r_+^2}.
\end{equation} Noting that $\ell_2 = z_2 - z_1 = (r_+^2 + \alpha^2)/r_+$, we have thus verified the identifies 
\begin{equation}
Q_{\partial_\psi}[C_2] = 2\beta \ell_2, \qquad Q_{\partial_\chi}[C_2] = -16 \pi j. 
\end{equation} Hence \eqref{Kerr1}, \eqref{Kerr2}, \eqref{Kerr3} are satisfied. 

\subsection{Chen-Teo instanton on $\mathbb{CP}^2 \setminus \mathbb{S}^1$} Consider the local family of Ricci-flat metrics~\cite{Chen:2015vva}
\begin{equation}\label{g}
\mathbf{g} = \frac{F(x,y)}{(x-y)H(x,y)} \left( \td\bar{\tau} + \frac{G(x,y)}{F(x,y)} \td\bar\phi\right)^2 + \frac{\kappa H(x,y)}{(x-y)^3} \left( \frac{\td x^2}{X(x)} - \frac{\td y^2}{Y(y)} - \frac{X(x)Y(y)}{\kappa F(x,y)} \td\bar\phi^2\right)
\end{equation} where $X = P(x), Y = P(y)$ are quartic polynomials 
\begin{equation}
P(u) = a_0 + a_1 u + a_2 u^2 + a_3 u^3 + a_4 u^4
\end{equation} and $F(x,y), H(x,y), G(x,y)$ are polynomials in $x,y$ given by
\begin{equation}
\begin{aligned}
F(x,y) & = y^2 X - x^2 Y\\
H(x,y) & =(\nu x + y)\left[(\nu x - y)(a_1 - a_3 x y) - 2(1 - \nu)(a_0 -a_4 x^2 y^2)\right] \\
G(x,y) & = (\nu^2 a_0 + 2 \nu a_3 y^3 + 2 \nu a_4 y^4 - a_4 y^4)X + (a_0 - 2\nu a_0 - 2\nu a_1 x - \nu^2 a_4 x^4) Y
\end{aligned}
\end{equation} The local metrics have 7 parameters: $(a_i, \nu, \kappa)$.

The Chen-Teo AF instanton~\cite{Chen:2011tc} is obtained by restricting to a certain two-parameter subfamily chosen to produce an AF smooth metric on $\mathbb{CP}^2 \setminus \mathbb{S}^1$, free of orbifold and conical singularities. This imposes a set of algebraic constraints on the $(a_i, \nu)$ which can be solved by setting~\cite{Chen:2015vva}: 
\begin{equation*}
\begin{aligned}
a_0 &= 4(1-\xi)(2\xi -1)(1 - 2\xi + 2\xi^2)\xi^4, \qquad a_1 = -\xi(1 - 4\xi + 10 \xi^2 - 20\xi^3 + 20\xi^4 - 16\xi^5 + 8 \xi^6) \\
a_2 &= -1 + 3\xi - 2 \xi^2 + 6 \xi^3 - 4 \xi^4, \qquad a_3 = 1, \qquad a_4 =0, \qquad \nu = -2 \xi^2
\end{aligned}
\end{equation*}  This has the result of explicitly factoring  $P(x)$ so it has roots $x_1 < x_2 < x_3$ given by
\begin{equation*}
x_1 = -4\xi^3(1 - \xi), \qquad x_2 = -\xi(1 - 2\xi + 2\xi^2), \qquad x_3 = 1- 2 \xi.
\end{equation*}   The resulting metric is parameterised by a scale $\kappa > 0$ and a parameter $\xi \in (1/2, 1/ \sqrt{2})$ \cite{Baird:2020psu}. The coordinate range is the rectangle in the $(x,y)$ plane with $x_1 < y < x_2 < x < x_3$ and the asymptotic region corresponds to $x \to x_2^+, y \to x_2^-$.   To exhibit asymptotic flatness change coordinates to 
\be
x= x_2- \frac{x_2 \sqrt{(1-v^2)\kappa} \cos^2\left(\frac{\theta}{2}\right)}{r}, \qquad y= x_2+ \frac{x_2 \sqrt{(1-v^2)\kappa} \sin^2\left(\frac{\theta}{2}\right)}{r},
\ee
and 
\be
\bar{\tau}= \sqrt{1-v^2} \tau+ b \phi, \qquad \bar{\phi}= c \phi
\ee
where 
\be
b =- \frac{8 \sqrt{\kappa} \xi^3 \left(8 \xi^5-16 \xi^4+12 \xi^3-8 \xi^2+4 \xi-1\right)}{(1-\xi) (2 \xi-1) \left(1-2 \xi^2\right)^2} \; , \qquad c =- \frac{2 \sqrt{\kappa} \left(2 \xi^2-2 \xi+1\right)}{(1-\xi) (2 \xi-1) \left(2 \xi^2-1\right)^2} \; .
\ee
Then one finds the metric is AF as $r \to \infty$ according to our Definition \ref{def:AF} (note these are not the $(r, \theta)$ coordinates in Proposition \ref{prop:AF}).  The  associated asymptotic invariants are found to be
\begin{equation}
m =  \frac{\sqrt{\kappa} (1 + 2 \xi^2)^2}{2\sqrt{1 - 4 \xi^4}}, \qquad j = - \frac{\kappa \xi^2 (1 - \xi + 2\xi^2)}{(1-\xi)(2\xi-1)} .
\end{equation}
  The canonical orbit space coordinates $(\rho,z)$ are given in terms of the coordinates $(x,y)$ by 
\begin{equation}
\rho =\frac{ \sqrt{-XY}}{(x-y)^2} \cdot \left(- c \sqrt{1-v^2} \right)
\end{equation} and 
\begin{equation}\label{zCT}
z = \frac{2(a_0 + a_2 x y + a_4 x^2 y^2) + (x+y)(a_1 + a_3 x y)}{2(x-y)^2} \cdot c \sqrt{1-v^2}
\end{equation} The form $\td y \wedge \td x$   is taken to have positive orientation on $\hat{M}$ in order to be consistent with our conventions.

 Let $(\psi, \chi)$ be $2\pi-$periodic angles defined by (\ref{basis}) where $(\tau, \phi)$ are the  coordinates in which the metric takes the canonical form \eqref{model} at infinity.  These are related by
\begin{align}
\bar\tau = \frac{b_1}{k_1} \psi + \frac{b_2}{k_2} \chi, \qquad \bar\phi = \frac{\psi}{k_1} + \frac{\chi}{k_2} ,
\end{align} where
\begin{equation}
\beta = \frac{16\pi \sqrt{k} \xi^4}{\sqrt{1-4\xi ^4} \left(2\xi ^2-2 \xi +1\right)^2}, \qquad \Omega = \frac{(1-\xi)^2 \sqrt{1 - 4 \xi^4}}{2 \sqrt{\kappa} \xi^2},
\end{equation} and 
\begin{gather}
b_1  = \frac{\xi ^2 \left(1-2 \xi 
   \left(8 \xi ^4-16 \xi ^3+10
   \xi ^2-3 \xi
   +2\right)\right)}{1-\xi } , \qquad b_2 = \frac{4 \xi ^3 \left(4 (1-\xi
   )^2 \xi  \left(2 \xi
   ^2+1\right)-1\right)}{1-2
   (1-\xi ) \xi }, \\
   k_1  = \frac{(1-2 \xi ) \left(1-2 \xi
   ^2\right)^2 \left(2 \xi
   ^2-2 \xi +1\right)}{8
   \sqrt{\kappa} (1-\xi ) \xi ^2}, \qquad k_2 = \frac{(1-2 \xi ) (1-\xi )
   \left(1-2 \xi
   ^2\right)^2}{2 \sqrt{\kappa}
   \left(2 \xi ^2-2 \xi
   +1\right)} .
\end{gather}  
The corner points $z_1 < z_2 < z_3$ are given as follows
\begin{equation}
z_1 = z(x_2, x_1), \quad  z_2 = z(x_3, x_1), \quad  z_3 = z(x_3, x_2).
\end{equation} where we have written \eqref{zCT} as $z = z(x,y)$.  We then have the following rod structure in the $(\partial_\psi, \partial_\chi)$ basis:
\begin{enumerate}
\item $I_1$:  $z < z_1$: $x = x_2$, $x_1 < y < x_2$; $\underline{v}_1=(0,1)$
\item $I_2$:  $z_1 < z < z_2$: $y = x_1$, $x_2 < x < x_3$;  $\underline{v}_2=(1,0)$
\item $I_3$: $z_2 < z < z_3$: $x = x_3$, $x_1 < y < x_2$;  $\underline{v}_3=(1,-1)$
\item $I_4$: $z > z_3$: $y = x_2$, $x_2 < x < x_3$;  $\underline{v}_4=(0,1)$
\end{enumerate} The two finite rods $I_2, I_3$ correspond to two non-trivial 2-cycles $C_2, C_3$ respectively. 

We now turn to computing the charges $Q_{\eta_i}[C_A]$ . Over the first finite rod, an involved calculation using (\ref{dY}) in the $(\partial_\psi, \partial_\chi)$ basis yields
\begin{equation}
Y(z_2) - Y(z_1) = \begin{pmatrix}  0 & 0 \\ -\frac{16 k \xi^4}{(2\xi-1)(1 - 2\xi^2)(1 - 2\xi + 2 \xi^2)} & -\frac{8 k \xi^2}{(2\xi -1)(1-2\xi^2)} \end{pmatrix}.
\end{equation} Note that the first row vanishes as it must because $\partial_\psi$ vanishes on $I_2$.  A similar computation gives
\begin{equation}
Y(z_3) - Y(z_2) = \begin{pmatrix}  -\frac{32 k \xi^6}{(1-2\xi^2)(1 - 2\xi + 2\xi^2)} & \frac{16 k \xi^5}{(1-\xi)(1 - 2\xi^2)} \\   -\frac{32 k \xi^6}{(1 - 2\xi^2)(1 - 2\xi + 2 \xi^2)} &  \frac{16 k \xi^5}{(1 - \xi)(1 - 2 \xi^2)} 
\end{pmatrix} .
\end{equation} We take the change of basis matrices 
\begin{equation}
L_2 = \begin{pmatrix}  0 & 1 \\ 1 & 0 \end{pmatrix}, \qquad L_3 = \begin{pmatrix} 1 & 0 \\ 1 & -1 \end{pmatrix}. 
\end{equation} This yields using \eqref{Q}
\begin{equation}
Q_{\partial_\psi}[C_2] = \frac{32\pi k \xi^4}{(2\xi-1)(1 - 2\xi^2)(1 - 2\xi + 2 \xi^2)}, \qquad Q_{\partial_\chi}[C_2] = \frac{16\pi k \xi^2}{(2\xi -1)(1-2\xi^2)},
\end{equation} and 
\begin{equation}
Q_{\partial_\psi}[C_3] = \frac{64\pi k \xi^6}{(1-2\xi^2)(1 - 2\xi + 2\xi^2)}, \qquad Q_{\partial_\chi}[C_3] =- \frac{32\pi k \xi^5}{(1-\xi)(1 - 2\xi^2)} .
\end{equation}  The identities \eqref{lA},  \eqref{CTm}, and \eqref{CTj} can now be verified.

\end{document}